\documentclass[a4paper,12pt]{amsart}

\usepackage{bbm}
\usepackage{units}
\usepackage{mathtools}
\usepackage{cancel}
\usepackage{amsmath}
\usepackage{amsfonts}
\usepackage{latexsym}
\usepackage{amssymb}
\usepackage{mathrsfs}
\usepackage{amscd}
\usepackage{hyperref}
\usepackage{soul}
\usepackage{psfrag}
\usepackage{graphicx}
\usepackage{ulem}
\usepackage{fullpage}
\usepackage[all]{xy}
\usepackage{rotating}
\usepackage{url}
\usepackage{color}
\usepackage{enumerate}
\usepackage{cleveref}
\usepackage{dsfont}
\usepackage{tikz}
\usetikzlibrary{arrows,backgrounds,arrows.meta,decorations.markings}
\usepackage{tikz-cd}

\numberwithin{equation}{section}
\numberwithin{figure}{section}

\theoremstyle{plain}
\newtheorem{thm}{Theorem}[section]

\theoremstyle{plain}
\newtheorem{lem}[thm]{Lemma}

\theoremstyle{remark}
\newtheorem{rem}[thm]{Remark}

\theoremstyle{plain}
\newtheorem{cor}[thm]{Corollary}

\theoremstyle{definition}
\newtheorem{defn}[thm]{Definition}

\theoremstyle{definition}

\theoremstyle{definition}

\theoremstyle{plain}
\newtheorem{prop}[thm]{Proposition}

\theoremstyle{plain}

\theoremstyle{definition}

\theoremstyle{plain}

\newcommand{\comments}[1]{}
\newcommand{\ra}{\rightarrow}
\newcommand{\rab}{\rangle}
\newcommand{\lra}{\longrightarrow}
\newcommand{\la}{\leftarrow}
\newcommand{\hookuparrow}{\mathrel{\rotatebox[origin=c]{90}{$\hookrightarrow$}}}
\newcommand{\hookdownarrow}{\mathrel{\rotatebox[origin=c]{-90}{$\hookrightarrow$}}}
\newcommand{\lab}{\langle}
\newcommand{\mbb}{\mathbb}
\newcommand{\mcal}{\mathcal}
\newcommand{\mfrak}{\mathfrak}
\newcommand{\N}{\mathbb N}
\newcommand{\Z}{\mathbb Z}
\newcommand{\C}{\mathbb{C}}
\newcommand{\R}{\mathbb{R}}
\newcommand{\F}{\mathbb{F}}
\newcommand{\mscr}{\mathscr}
\newcommand{\vlon}{\varepsilon}
\newcommand{\elon}{\epsilon}
\newcommand{\bplus}{\bigoplus}
\newcommand{\uset}{\underset}
\newcommand{\oset}{\overset}
\newcommand{\oline}{\overline}
\newcommand{\vphi}{\varphi}
\newcommand{\ot}{\otimes}
\newcommand{\btimes}{\boxtimes}
\newcommand{\wt}{\widetilde}
\newcommand*{\LargerCdot}{\raisebox{-0.25ex}{\scalebox{1.2}{$\cdot$}}}
\newcommand{\ID}{\textbf{I}_{\mcal D}}
\newcommand{\IC}{\textbf{I}_{\mcal C}}
\newcommand{\W}{\textbf{W}}
\newcommand{\Irr}{\text{Irr}}
\newcommand{\So}{\textbf{S}}
\newcommand{\Hilb}{\textbf{Hilb}}

\title{Q-system completion of 2-functors}
\author{Mainak Ghosh}
\newcommand{\Contact}{{
		\bigskip
		\footnotesize
				
		Mainak Ghosh, \textsc{Stat-Math Unit, Indian Statistical Institute}\par\nopagebreak
		\textit{E-mail address}: \texttt{main\_ghosh@rediffmail.com}
		
	}}
\begin{document}
	\maketitle
	\global\long\def\vlon{\varepsilon}
	\global\long\def\bt{\bowtie}
	\global\long\def\ul#1{\underline{#1}}
	\global\long\def\ol#1{\overline{#1}}
	\global\long\def\norm#1{\left\|{#1}\right\|}
	\global\long\def\os#1#2{\overset{#1}{#2}}
	\global\long\def\us#1#2{\underset{#1}{#2}}
	\global\long\def\ous#1#2#3{\overset{#1}{\underset{#3}{#2}}}
	\global\long\def\t#1{\text{#1}}
	\global\long\def\lrsuf#1#2#3{\vphantom{#2}_{#1}^{\vphantom{#3}}#2^{#3}}
	\global\long\def\tr{\triangleright}
	\global\long\def\tl{\triangleleft}
	\global\long\def\cc90#1{\begin{sideways}#1\end{sideways}}
	\global\long\def\turnne#1{\begin{turn}{45}{#1}\end{turn}}
	\global\long\def\turnnw#1{\begin{turn}{135}{#1}\end{turn}}
	\global\long\def\turnse#1{\begin{turn}{-45}{#1}\end{turn}}
	\global\long\def\turnsw#1{\begin{turn}{-135}{#1}\end{turn}}
	\global\long\def\fusion#1#2#3{#1 \os{\textstyle{#2}}{\otimes} #3}
	
	\global\long\def\abs#1{\left|{#1}\right	|}
	\global\long\def\red#1{\textcolor{red}{#1}}
\begin{abstract}
	A Q-system is a unitary version of a separable Frobenius algebra object in a C*-tensor category or a C*-2-category. We prove that, for C*-2-categories $\mcal C$ and $\mcal D$, the C*-2-category $\textbf{Fun}(\mcal C, \mcal D)$ of $ * $-$ 2 $-functors, $ * $-$ 2 $-transformations and $ * $-$ 2 $-modifications is Q-system complete, whenever $\mcal D$ is Q-system complete. We use this result to provide a characterisation of Q-system complete categories in terms of $ * $-$ 2 $-functors and to prove that the $ 2 $-category of actions of a unitary fusion category $\mcal C$ on C*-algebras is Q-system complete.
\end{abstract}

\section{Introduction}
The modern theory of subfactors fostered with V. Jones' landmark results in \cite{J83}. The standard invariant of a finite index subfactor of a $II_1$ factor was first defined as a $\lambda$-lattice \cite{P95}. S. Popa proved that amenable finite subfactors of the hyperfinite $II_1$ factor are completely classified by their standard invariant \cite{P95}, which are axiomatized in general by $\lambda$-lattices \cite{P94,P95} or planar algebras \cite{J99}. In the finite depth setting, A. Ocneanu introduced and established the theory of biunitary connection on 4-partite graph as an essential tool for constructing hyperfinite subfactors. Biunitary connections feature in his paragroup axiomatization of finite depth standard invariants \cite{O88,EvKaw} but can also be used to construct infinite depth hyperfinite subfactors from finite graphs (see \cite{P89,Sch}). While the other approaches to standard invariants are now more common, the theory of biunitary connections remains an important ingredient in the construction and classification of hyperfinite subfactors \cite{EvKaw,JMS}.
Many features of subfactor theory now have a clear higher-categorical interpretation \cite{M03,B97,CPJP,JMS,DGGJ}.

\vspace*{2mm}

In \cite{M03}, a Q-system which is a unitary version of a separable Frobenius algebra object in a C*-tensor category or a C*-2-category, is exhibited as an alternative axiomatization of the standard invariant of a finite index subfactor \cite{O88,P95,J99}. Q-systems were first introduced in \cite{Lon} to characterize canonical endomorphism associated to a finite index subfactor of an infinite factor. In the context of C*-2-categories, a Q-system is a 1-cell $_bQ_b \in \mcal C_1(b,b)$ along with two 2-cells $m : Q \boxtimes Q \to Q$ (multiplication) and $i : 1_b \to Q$ (unit), which are graphically denoted by the following:
\[m = \raisebox{-6mm}{
	\begin{tikzpicture}
		\draw[red,in=90,out=90,looseness=2] (-0.5,0.5) to (-1.5,0.5);
		\node at (-1,1.1) {${\color{red}\bullet}$};
		\draw[red] (-1,1.1) to (-1,1.6);
		\node[left,scale=0.7] at (-1,1.4) {$Q$};
		\node[left,scale=0.7] at (-1.6,0.5) {$Q$};
		\node[right,scale=0.7] at (-.5,.5) {$Q$};
\end{tikzpicture}} \ \ \ \ \ \ i = \raisebox{-6mm}{
	\begin{tikzpicture}
		\draw [red] (-0.8,-.6) to (-.8,.6);
		\node at (-.8,-.6) {${\color{red}\bullet}$};
		\node[left,scale=0.7] at (-.8,0) {$Q$};
\end{tikzpicture}} \ \ \ \ \ \ m^* = \raisebox{-6mm}{
	\begin{tikzpicture}
		\draw[red,in=-90,out=-90,looseness=2] (-0.5,0.5) to (-1.5,0.5);
		\node at (-1,-.1) {${\color{red}\bullet}$};
		\draw[red] (-1,-.1) to (-1,-.6);
		\node[left,scale=0.7] at (-1,-.4) {$Q$};
		\node[left,scale=0.7] at (-1.6,0.5) {$Q$};
		\node[right,scale=0.7] at (-.5,.5) {$Q$};
\end{tikzpicture}} \ \ \ \ \ \ i^* = \raisebox{-6mm}{
	\begin{tikzpicture}
		\draw [red] (-0.8,-.6) to (-.8,.6);
		\node at (-.8,.6) {${\color{red}\bullet}$};
		\node[left,scale=0.7] at (-.8,0) {$Q$};
\end{tikzpicture}} \]
These $2$-cells satisfy the following:
\[\raisebox{-6mm}{
	\begin{tikzpicture}
		\draw[red,in=90,out=90,looseness=2] (0,0) to (1,0);
		\draw[red,in=90,out=90,looseness=2] (0.5,.6) to (-.5,.6);
		\draw[red] (-.5,.6) to (-.5,0);
		\node at (.5,.6) {$\red{\bullet}$};
		\node at (0,1.2) {$\red{\bullet}$};
		\draw[red] (0,1.2) to (0,1.6);
\end{tikzpicture}}
= \raisebox{-6mm}{\begin{tikzpicture}
		\draw[red,in=90,out=90,looseness=2] (0,0) to (1,0);
		\draw[red,in=90,out=90,looseness=2] (.5,.6) to (1.5,.6);
		\node at (.5,.6) {$\red{\bullet}$};
		\draw[red] (1.5,.6) to (1.5,0);
		\node at (1,1.2) {$\red{\bullet}$};
		\draw[red] (1,1.2) to (1,1.6);
\end{tikzpicture}} \ \t{(associativity)} \ \ \ \ \ \ \raisebox{-4mm}{
	\begin{tikzpicture}
		\draw[red,in=90,out=90,looseness=2] (0,0) to (1,0);
		\node at (.5,.6) {$\red{\bullet}$};
		\node at (0,0) {$\red{\bullet}$};
		\draw[red] (.5,.6) to (.5,1.2);
\end{tikzpicture}}=
\raisebox{-4mm}{
	\begin{tikzpicture}
		\draw[red,in=90,out=90,looseness=2] (0,0) to (1,0);
		\node at (.5,.6) {$\red{\bullet}$};
		\node at (1,0) {$\red{\bullet}$};
		\draw[red] (.5,.6) to (.5,1.2);
\end{tikzpicture}}
=
\raisebox{-2mm}{
	\begin{tikzpicture}
		\draw[red] (0,0) to (0,1.2);
\end{tikzpicture}} \ \t{(unitality)} \]

\[\raisebox{-8mm}{
	\begin{tikzpicture}
		\draw[red,in=90,out=90,looseness=2] (0,0) to (1,0);
		\draw[red,in=-90,out=-90,looseness=2] (1,0) to (2,0);
		\node at (.5,.6) {$\red{\bullet}$};
		\node at (1.5,-.6) {$\red{\bullet}$};
		\draw[red] (.5,.6) to (.5,1.2);
		\draw[red] (1.5,-.6) to (1.5,-1.2);
		\draw[red] (0,0) to (0,-.6);
		\draw[red] (2,0) to (2,.6);
\end{tikzpicture}} =
\raisebox{-6mm}{
	\begin{tikzpicture}
		\draw[red,in=90,out=90,looseness=2] (0,0) to (1,0);
		\node at (.5,.6) {$\red{\bullet}$};
		\draw[red] (.5,.6) to (.5,1.2);
		\draw[red,in=-90,out=-90,looseness=2] (0,1.8) to (1,1.8);
		\node at (.5,1.2) {$\red{\bullet}$};
\end{tikzpicture}} =
\raisebox{-8mm}{
	\begin{tikzpicture}
		\draw[red,in=-90,out=-90,looseness=2] (0,0) to (1,0);
		\draw[red,in=90,out=90,looseness=2] (1,0) to (2,0);
		\node at (.5,-.6) {$\red{\bullet}$};
		\node at (1.5,.6) {$\red{\bullet}$};
		\draw[red] (.5,-.6) to (.5,-1.2);
		\draw[red] (1.5,.6) to (1.5,1.2);
		\draw[red] (0,0) to (0,.6);
		\draw[red] (2,0) to (2,-.6);
\end{tikzpicture}} \ \t{(Frobenius condition)} \ \ \ \ \ \ \raisebox{-6mm}{
	\begin{tikzpicture}
		\draw[red,in=90,out=90,looseness=2] (0,0) to (1,0);
		\draw[red,in=-90,out=-90,looseness=2] (0,0) to (1,0);
		\node at (.5,-.6) {$\red{\bullet}$};
		\node at (.5,.6) {$\red{\bullet}$};
		\draw[red] (.5,-.6) to (.5,-1.2);
		\draw[red] (.5,.6) to (.5,1.2);
\end{tikzpicture}} = 
\raisebox{-6mm}{
	\begin{tikzpicture}
		\draw[red] (0,0) to (0,2.4);
\end{tikzpicture}} \ \t{(Separability)} \]

\vspace*{2mm}

In \cite{CPJP} the notion of \textit{Q-system completion} for C*/W*-2-categories, which is another version of a higher idempotent completion for C*/W*-2-categories in comparison with $ 2 $-categories of separable monads \cite{DR18} and condensation monads \cite{GJF19}, has been introduced, to induce actions of unitary fusion categories on C*-algebras. In the recent years, \textit{Q-system completion} has gained considerable interest \cite{CP,CPJ,G}
 
Given a C*/W*-2-category $\mcal C$, which is also locally orthogonal projection complete, it's \textit{Q-system completion} is the 2-category $\textbf{QSys}(\mcal C)$ of Q-systems, bimodules and intertwiners in $\mcal C$ . There is a canonical inclusion *-2-functor $\iota_{\mcal C} : \mcal C \hookrightarrow \textbf{QSys}(\mcal C)$ which is always an equivalence on all hom categories. $\mcal C$ is said to be \textit{Q-system complete} if $\iota_{\mcal C}$ is a *-equivalence of *-2-categories. It has been established in \cite{CPJP} that $\mcal C$ is \textit{Q-system complete} if and only if every Q-system `splits' in $\mcal C$.

\vspace*{2mm}

In this paper, we study Q-system completeness of the $ 2 $-category $\textbf{Fun}(\mcal C, \mcal D)$ of $ * $-$ 2 $-functors, $ * $-$ 2 $-transformations, and $ 2 $-modifications,  where $\mcal C$ and $\mcal D$ are strict C*-2-categories. The following is the main theorem of the paper.

\begin{thm}\label{maintheorem}
	Suppose $\mcal C$, $\mcal D$ are strict C*-2-categories and $\mcal D$ is Q-system complete. Then $\normalfont\textbf{Fun}(\mcal C, \mcal D)$ is Q-system complete.
\end{thm}   
Using the theorem, we provide a provide a characterisation of \textit{Q-system complete} categories in terms of $ 2 $-functors. In \cite{CPJ}, the $ 2 $-category of actions of a unitary fusion category $\mcal C$ on C*-algebras $\textbf{C*Alg}_{\mcal C}$, has been introduced to study inductive limit actions of fusion categories on AF-algebras. As a consequence of \Cref{maintheorem}, we have the following result.

\begin{cor}\label{C*Algcor}
	$\normalfont \textbf{C*Alg}_{\mcal C}$ is Q-system complete
\end{cor}

If $\mcal C , \mcal D$ are C*-2-categories, then it has been shown in \cite{CP} that $\textbf{Fun}(\mcal C, \mcal D)$ is also a C*-2-category. In \Cref{prelim}, we show that $\textbf{Fun}(\mcal C, \mcal D)$ is locally orthogonal projection complete when $\mcal D$ is so. Let $\left(\psi_\bullet, m_\bullet , i_\bullet \right)$ be a Q-system in $\t{End}(F)$ for some $*$-2-functor $F : \mcal C \to \mcal D$ where, $m : \psi \otimes \psi \Rrightarrow \psi$ is the multiplication map and $i : 1_{F} \Rrightarrow \psi$ is the unit map. This means that, for every $a \in \mcal C_0$ we have a Q-system $\left(\psi_a,m_a,i_a\right)$ in $\t{End}_{\mcal D_1} (Fa)$. To prove that $\textbf{Fun}(\mcal C, \mcal D)$ is Q-system complete (assuming $\mcal D$ is Q-system complete) we need to show that there is a $1$-cell $\phi \in \textbf{Fun}(\mcal C,\mcal D)_1 (F,G)$ , for some $*$-2-functor $G : \mcal C \to \mcal D$, having a unitarily separable dual $\ol \phi \in \textbf{Fun}(\mcal C,\mcal D)_1(G,F)$ such that the $1$-cells $\psi$ and $\ol \phi \otimes \phi$ are isomorphic as Q-systems. In order to achieve this the following tasks need to be completed:
\begin{itemize}
	\item[(1)] We need to find a suitable $*$-2-functor $G : \mcal C \to \mcal D \ $.
	\item[(2)] To find a dualizable $1$-cell $\phi \in \textbf{Fun}(\mcal C,\mcal D)_1 (F,G) \ $.
	\item[(3)] To find a 2-modification $\gamma : \ol \phi \otimes \phi \Rrightarrow \psi$, which is a unitary and intertwines the algebra maps.
\end{itemize}
We proceed with $(1)$ in \Cref{2functor}. \Cref{2transformation} is regarding $(2)$. In \Cref{2modification}, we establish $(3)$ and discuss the consequences of \Cref{maintheorem} 

\subsection*{Acknowledgements}
The author would like to thank Shamindra Kumar Ghosh, Corey Jones and David Penneys for several fruitful discussions.

\section{Preliminaries}\label{prelim}
In this section we will furnish the necessary background on \textit{Q-system completion} and the 2-category of $ * $-$ 2 $-functors $\textbf{Fun}(\mcal C, \mcal D)$, where $\mcal C, \mcal D$ are C*-2-categories .

\subsection{Notations related to 2-categories}\label{graphcalc}
We refer the reader to \cite{JY21} for basics of 2-categories.\\
Suppose $\mcal C$ is a 2-category and  $a,b \in \mcal C_0$ be two $0$-cells. A $1$-cell from $a  \xrightarrow{X} b$ is denoted by $_bX_a$. Pictorially, a $1$-cell will be denoted by a strand and a $2$-cell will be denoted by a box with strings with passing through it. Suppose we have two $1$-cells $X, Y \in \mcal C_1(a,b)$ and $f \in \mcal C_2(X,Y)$ be a $2$-cell. Then we will denote $f$ as 
\raisebox{-11mm}{
	\begin{tikzpicture}
		\draw (0,.8) to (0,-.8);
		\node[draw,thick,rounded corners,fill=white, minimum width=20] at (0,0) {$f$};
		\node[right] at (0,-.7) {$X$};
		\node[right] at (0,.7) {$Y$};
\end{tikzpicture}}
We write tensor product $\boxtimes$ of $1$-cells from right to left $ _cY \us{b}\boxtimes X_a$.
$ 1 $-cells which are Q-systems will be denoted by red strands.
The notion of C*-2-categories is believed to first appear in \cite{LR}. For basics of C*/W*-2-categories we refer the reader to \cite{CPJP,GLR85}.

\subsection{Q-system completion}\

Let $\mcal C$ be a C*-2-category.
\begin{defn} \label{Qsysdefn}
	A Q-system in $\mcal C$ is a $1$-cell $_bQ_b \in \mcal C_1(b,b)$ along with multiplication map $m \in \mcal C_2(Q \boxtimes_b Q, Q)$ and unit map $i \in \mcal C_2(1_b,Q)$, as mentioned in Section 1, satisfying the following properties:
\end{defn}
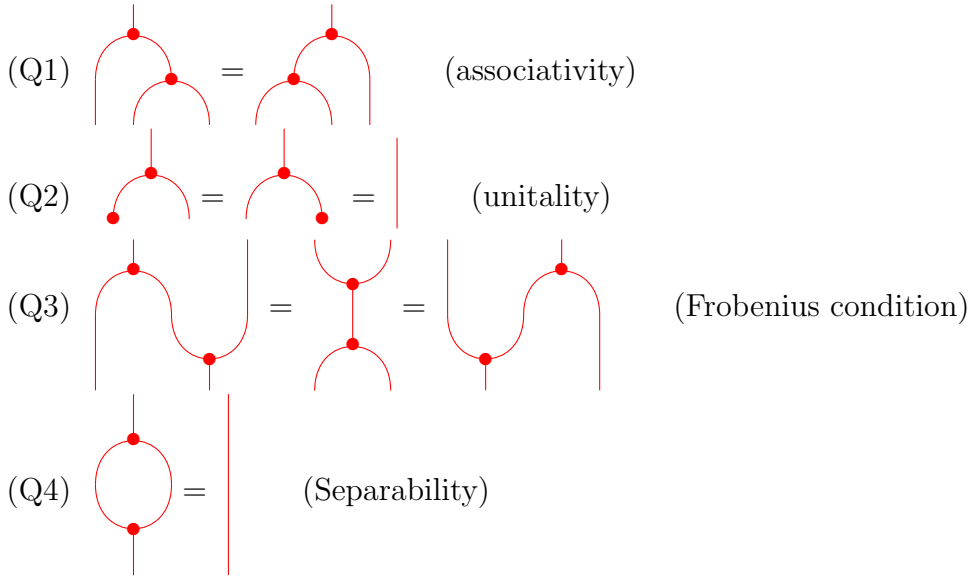
\begin{itemize}
	\item[(Q1)]
	\raisebox{-6mm}{
		\begin{tikzpicture}
			\draw[red,in=90,out=90,looseness=2] (0,0) to (1,0);
			\draw[red,in=90,out=90,looseness=2] (0.5,.6) to (-.5,.6);
			\draw[red] (-.5,.6) to (-.5,0);
			\node at (.5,.6) {$\red{\bullet}$};
			\node at (0,1.2) {$\red{\bullet}$};
			\draw[red] (0,1.2) to (0,1.6);
	\end{tikzpicture}}
	= \raisebox{-6mm}{\begin{tikzpicture}
			\draw[red,in=90,out=90,looseness=2] (0,0) to (1,0);
			\draw[red,in=90,out=90,looseness=2] (.5,.6) to (1.5,.6);
			\node at (.5,.6) {$\red{\bullet}$};
			\draw[red] (1.5,.6) to (1.5,0);
			\node at (1,1.2) {$\red{\bullet}$};
			\draw[red] (1,1.2) to (1,1.6);
	\end{tikzpicture}} \ \ \ \ \ \ \t{(associativity)} 
	
	\item[(Q2)]
	\raisebox{-4mm}{
		\begin{tikzpicture}
			\draw[red,in=90,out=90,looseness=2] (0,0) to (1,0);
			\node at (.5,.6) {$\red{\bullet}$};
			\node at (0,0) {$\red{\bullet}$};
			\draw[red] (.5,.6) to (.5,1.2);
	\end{tikzpicture}} = 
	\raisebox{-4mm}{
		\begin{tikzpicture}
			\draw[red,in=90,out=90,looseness=2] (0,0) to (1,0);
			\node at (.5,.6) {$\red{\bullet}$};
			\node at (1,0) {$\red{\bullet}$};
			\draw[red] (.5,.6) to (.5,1.2);
	\end{tikzpicture}}
	= 
	\raisebox{-3mm}{
		\begin{tikzpicture}
			\draw[red] (0,0) to (0,1.2);
	\end{tikzpicture}} \ \ \ \ \ \ \t{(unitality)} 
	
	\item[(Q3)]
	\raisebox{-10mm}{
		\begin{tikzpicture}
			\draw[red,in=90,out=90,looseness=2] (0,0) to (1,0);
			\draw[red,in=-90,out=-90,looseness=2] (1,0) to (2,0);
			\node at (.5,.6) {$\red{\bullet}$};
			\node at (1.5,-.6) {$\red{\bullet}$};
			\draw[red] (.5,.6) to (.5,1);
			\draw[red] (1.5,-.6) to (1.5,-1);
			\draw[red] (0,0) to (0,-1);
			\draw[red] (2,0) to (2,1);
	\end{tikzpicture}} \ =
	\raisebox{-10mm}{
		\begin{tikzpicture}
			\draw[red,in=90,out=90,looseness=2] (0,0) to (1,0);
			\node at (.5,.6) {$\red{\bullet}$};
			\draw[red] (.5,.6) to (.5,1.4);
			\draw[red,in=-90,out=-90,looseness=2] (0,2) to (1,2);
			\node at (.5,1.4) {$\red{\bullet}$};
	\end{tikzpicture}} =
	\raisebox{-10mm}{
		\begin{tikzpicture}
			\draw[red,in=-90,out=-90,looseness=2] (0,0) to (1,0);
			\draw[red,in=90,out=90,looseness=2] (1,0) to (2,0);
			\node at (.5,-.6) {$\red{\bullet}$};
			\node at (1.5,.6) {$\red{\bullet}$};
			\draw[red] (.5,-.6) to (.5,-1);
			\draw[red] (1.5,.6) to (1.5,1);
			\draw[red] (0,0) to (0,1);
			\draw[red] (2,0) to (2,-1);
	\end{tikzpicture}} \ \ \ \ \ \ \t{(Frobenius condition)} 
	
	\item[(Q4)]
	\raisebox{-10mm}{
		\begin{tikzpicture}
			\draw[red,in=90,out=90,looseness=2] (0,0) to (1,0);
			\draw[red,in=-90,out=-90,looseness=2] (0,0) to (1,0);
			\node at (.5,-.6) {$\red{\bullet}$};
			\node at (.5,.6) {$\red{\bullet}$};
			\draw[red] (.5,-.6) to (.5,-1.2);
			\draw[red] (.5,.6) to (.5,1.2);
	\end{tikzpicture}} = 
	\raisebox{-10mm}{
		\begin{tikzpicture}
			\draw[red] (0,0) to (0,2.4);
	\end{tikzpicture}} \ \ \ \ \ \ \t{(Separability)}
	
\end{itemize}
\begin{rem}\label{Qdual}
	A Q-system $Q$ is a self-dual 1-cell with $ev_Q \coloneqq 
	\raisebox{-4mm}{
		\begin{tikzpicture}[scale=.8]
			\draw[red,in=-90,out=-90,looseness=2] (-0.5,0.5) to (-1.5,0.5);
			\node at (-1,-.1) {${\color{red}\bullet}$};
			\draw[red] (-1,-.1) to (-1,-.6);
			\node[left,scale=0.7] at (-1,-.4) {$Q$};
			\node[left,scale=0.7] at (-1.6,0.4) {$Q$};
			\node[right,scale=0.7] at (-.5,.4) {$Q$};
			\node at (-1,-.6) {${\color{red}\bullet}$};
	\end{tikzpicture}}
	$ and $coev_Q \coloneqq \raisebox{-4mm}
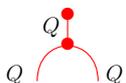
{
		\begin{tikzpicture}[scale=.8]
			\draw[red,in=90,out=90,looseness=2] (-0.5,0.5) to (-1.5,0.5);
			\node at (-1,1.1) {${\color{red}\bullet}$};
			\draw[red] (-1,1.1) to (-1,1.6);
			\node[left,scale=0.7] at (-1,1.4) {$Q$};
			\node[left,scale=0.7] at (-1.6,0.6) {$Q$};
			\node[right,scale=0.7] at (-.5,.6) {$Q$};
			\node at (-1,1.6) {${\color{red}\bullet}$};
	\end{tikzpicture}}$
\end{rem}
\begin{defn}
	Suppose $\mcal C$ is a C*-2-category and $_bX_a \in \mcal C_1(a,b)$. A \textit{unitarily separable left dual} for $_b X_a$ is a dual $\left( _a \ol X_b, ev_X, coev_X \right)$ such that $ev_X \circ ev_X^* = \t{id}_{1_a}$ (cf. \cite[Example 3.9]{CPJP}).
\end{defn}
Given a unitarily separable left dual for $_bX_a \in \mcal C_1(a,b)$, $_b X \us{a} \boxtimes \ol X_b \in \mcal C_1(b,b)$ is a Q-system with multiplication map $m \coloneqq \t{id}_X \boxtimes ev_X \boxtimes \t{id}_{\ol X}$ and unit map $i \coloneqq coev_X$.

Given a Q-system $Q \in \mcal C_1(b,b)$, if it is of the above form,  then we say that the Q-system $Q$ `splits'.

\begin{defn}
	Suppose $\mcal C$ is a C*-2-category. $\mcal C$ is said to be \textit{locally orthogonal projection complete} if, given any projection $p \in \t{End}(X)$ for a $ 1 $-cell $X \in \mcal C_1(a,b)$, there is another $ 1 $-cell $Y \in \mcal C_1(a,b)$ and an isometry $v \in \mcal C_2(Y,X)$ such that $v v^* = p$ .  
\end{defn}
Throughout the paper by a C*-2-category we will mean a C*-2-category which is locally orthogonal projection complete unless specified otherwise.

\begin{defn}\cite{CPJP}
	Suppose $P \in \mcal C_1(a,a)$ and $Q \in \mcal C_1(b,b)$ are Q-systems. A $Q-P$ bimodule is a triple $\left( X, \lambda_X, \rho_X \right)$ consisting of $_bX_a \in \mcal C_1(a,b)$, $\lambda_X \in \mcal C_2 \left(Q \boxtimes X, X\right)$ and $\rho_X \in \mcal C_2(X \boxtimes P, X)$ satisfying certain properties. We represent $X, P, Q$ graphically by purple, red and magenta strands respectively. We denote $\lambda_X, \rho_X, \lambda_X^*, \rho_X^*$ as follows:
	\[\lambda_X \ = \ \raisebox{-2mm}{\begin{tikzpicture}
			\draw[purple] (0,0) to (0,.6);
			\draw[red][in=120,out=90,looseness=1] (-.5,0) to (0,.3);
			\node[scale=.8] at (0,.3) {$\color{purple}{\bullet}$};
	\end{tikzpicture}} \ \ \ \ 
	\rho_X \ = \ \raisebox{-2mm}{\begin{tikzpicture}
			\draw[purple] (0,0) to (0,.6);
			\draw[magenta][in=60,out=90,looseness=1] (.5,0) to (0,.3);
			\node[scale=.8] at (0,.3) {$\color{purple}{\bullet}$};
	\end{tikzpicture}} \ \ \ \
	\lambda_X^* \ = \ \raisebox{-2mm}{\begin{tikzpicture}[rotate=180]
			\draw[purple] (0,0) to (0,.6);
			\draw[red][in=60,out=90,looseness=1] (.5,0) to (0,.3);
			\node[scale=.8] at (0,.3) {$\color{purple}{\bullet}$};
	\end{tikzpicture}} \ \ \ \
	\rho_X^* \ = \ \raisebox{-2mm}{\begin{tikzpicture}[rotate=180]
			\draw[purple] (0,0) to (0,.6);
			\draw[magenta][in=120,out=90,looseness=1] (-.5,0) to (0,.3);
			\node[scale=.8] at (0,.3) {$\color{purple}{\bullet}$};
	\end{tikzpicture}}  \]
	The bimodule axioms are as follows:
	\begin{itemize}
		\item[(B1)] \raisebox{-4mm}{\begin{tikzpicture}
				\draw[purple] (0,0) to (0,1.2);
				\draw[red][in=120,out=90,looseness=1] (-.5,0) to (0,.5);
				\draw[red][in=120,out=90,looseness=1] (-.7,0) to (0,.7);
				\node[scale=.8] at (0,.52) {$\color{purple}{\bullet}$};
				\node[scale=.8] at (0,.72) {$\color{purple}{\bullet}$};
				\node at (.6,.6) {$=$};
				\draw[red,in=90,out=90,looseness=2] (1,0) to (1.5,0);
				\node[scale=.7] at (1.25,.27) {$\red{\bullet}$};
				\draw[red,in=120,out=90,looseness=1] (1.25,.27) to (1.75,.57);
				\draw[purple] (1.75,0) to (1.75,1.2);
				\node[scale=.8] at (1.75,.59) {$\color{purple}{\bullet}$};
		\end{tikzpicture}} \hspace*{2mm}, \hspace*{2mm}
		\raisebox{-4mm}{\begin{tikzpicture}
				\draw[purple] (0,0) to (0,1.2);
				\draw[magenta][in=60,out=90,looseness=1] (.5,0) to (0,.5);
				\draw[magenta][in=60,out=90,looseness=1] (.7,0) to (0,.7);
				\node[scale=.8] at (0,.52) {$\color{purple}{\bullet}$};
				\node[scale=.8] at (0,.72) {$\color{purple}{\bullet}$};
				\node at (1,.6) {$=$};
				\draw[purple] (1.3,0) to (1.3,1.2);
				\draw[magenta,in=90,out=90,looseness=2] (1.55,0) to (2.05,0);
				\draw[magenta,in=60,out=90,looseness=1] (1.8,.27) to (1.3,.57);
				\node[scale=.7] at (1.8,.27) {$\color{magenta}{\bullet}$};
				\node[scale=.8] at (1.3,.57) {$\color{purple}{\bullet}$};
		\end{tikzpicture}} \hspace*{2mm} and \hspace*{2mm}
		\raisebox{-4mm}{\begin{tikzpicture}
				\draw[purple] (0,0) to (0,1);
				\draw[red][in=120,out=90,looseness=1] (-.5,0) to (0,.6);
				\draw[magenta][in=60,out=90,looseness=1] (.5,0) to (0,.3);
				\node[scale=.8] at (0,.6) {$\color{purple}{\bullet}$};
				\node[scale=.8] at (0,.3) {$\color{purple}{\bullet}$};
				\node at (.7,.6) {$=$};
				\draw[purple] (1.5,0) to (1.5,1);
				\draw[red][in=120,out=90,looseness=1] (1,0) to (1.5,.3);
				\draw[magenta][in=60,out=90,looseness=1] (2,0) to (1.5,.6);
				\node[scale=.8] at (1.5,.6) {$\color{purple}{\bullet}$};
				\node[scale=.8] at (1.5,.3) {$\color{purple}{\bullet}$};
		\end{tikzpicture}} \hspace*{2mm} (associativity)
		\item[(B2)] \raisebox{-4mm}{\begin{tikzpicture}
				\draw[purple] (0,0) to (0,1);
				\draw[red][in=120,out=90,looseness=1] (-.5,.2) to (0,.6);
				\node[scale=.7] at (-.5,.2) {$\red{\bullet}$};
				\node[scale=.8] at (0,.6) {$\color{purple}{\bullet}$};
				\node at (.4,.6) {$=$};
				\draw[purple] (.8,0) to (.8,1);
				\node at (1.2,.6) {$=$};
				\draw[purple] (1.6,0) to (1.6,1);
				\draw[magenta][in=60,out=90,looseness=1] (2.1,.2) to (1.6,.6);
				\node[scale=.7] at (2.1,.2) {$\color{magenta}{\bullet}$};
				\node[scale=.8] at (1.6,.6) {$\color{purple}{\bullet}$};
		\end{tikzpicture}} \hspace*{2mm} (unitality)
		
		\item[(B3)] \raisebox{-6mm}{\begin{tikzpicture}
				\draw[purple] (0,0) to (0,1.2);
				\draw[red][in=-120,out=-90,looseness=1] (-.3,.5) to (0,.3);
				\draw[red,in=90,out=90,looseness=2] (-.7,.5) to (-.3,.5);
				\draw[red] (-.7,.5) to (-.7,0);
				\draw[red] (-.5,.7) to (-.5,1);
				\node[scale=.7] at (-.5,.72) {$\red{\bullet}$};
				\node[scale=.8] at (0,.3) {$\color{purple}{\bullet}$};
				\node at (.4,.6) {$=$};
				\draw[purple] (1.2,0) to (1.2,1.2);
				\draw[red][in=-120,out=-90,looseness=1] (.8,1.2) to (1.2,.8);
				\draw[red][in=120,out=90,looseness=1] (.8,0) to (1.2,.4);
				\node[scale=.8] at (1.2,.8) {$\color{purple}{\bullet}$};
				\node[scale=.8] at (1.2,.4) {$\color{purple}{\bullet}$};
				\node at (1.6,.6) {$=$};
				\draw[purple] (2.8,0) to (2.8,1.2);
				\draw[red][in=120,out=90,looseness=1] (2.4,.4) to (2.8,.8);
				\draw[red,in=-90,out=-90,looseness=2] (2.4,.4) to (2,.4);
				\draw[red] (2,.4) to (2,1);
				\draw[red] (2.2,.2) to (2.2,-.1);
				\node[scale=.7] at (2.2,.15) {$\red{\bullet}$};
				\node[scale=.8] at (2.8,.8) {$\color{purple}{\bullet}$};
		\end{tikzpicture}} \hspace*{2mm} and \hspace*{2mm}
		\raisebox{-6mm}{\begin{tikzpicture}
				\draw[purple] (0,0) to (0,1.2);
				\draw[magenta][in=-60,out=-90,looseness=1] (.3,.5) to (0,.3);
				\draw[magenta][in=90,out=90,looseness=2] (.3,.5) to (.7,.5);
				\draw[magenta] (.5,.7) to (.5,1); 
				\draw[magenta] (.7,.5) to (.7,0);
				\node[scale=.8] at (0,.3) {$\color{purple}{\bullet}$};
				\node[scale=.7] at (.5,.7) {$\color{magenta}{\bullet}$};
				\node at (1.1,.6) {$=$};
				\draw[purple] (1.5,0) to (1.5,1.2);
				\draw[magenta][in=-60,out=-90,looseness=1] (1.9,1.2) to (1.5,.8);
				\draw[magenta][in=60,out=90,looseness=1] (1.9,0) to (1.5,.4);
				\node[scale=.8] at (1.5,.8) {$\color{purple}{\bullet}$};
				\node[scale=.8] at (1.5,.4) {$\color{purple}{\bullet}$};
				\node at (2.3,.6) {$=$};
				\draw[purple] (2.7,0) to (2.7,1.2);
				\draw[magenta][in=60,out=90,looseness=1] (3,.4) to (2.7,.7);
				\draw[magenta,in=-90,out=-90,looseness=2] (3,.4) to (3.4,.4);
				\draw[magenta] (3.4,.4) to (3.4,1);
				\draw[magenta] (3.2,.2) to (3.2,0);
				\node[scale=.7] at (3.2,.15) {$\color{magenta}{\bullet}$};
				\node[scale=.8] at (2.7,.7) {$\color{purple}{\bullet}$}; 
		\end{tikzpicture}} \hspace*{2mm} (frobenius condition)
		
		\item[(B4)] \raisebox{-6mm}{\begin{tikzpicture}
				\draw[purple] (0,0) to (0,1.2);
				\draw[red][in=-120,out=-90,looseness=1] (-.3,.5) to (0,.3);
				\draw[red][in=120,out=90,looseness=1] (-.3,.5) to (0,.8);
				\node[scale=.8] at (0,.3) {$\color{purple}{\bullet}$};
				\node[scale=.8] at (0,.8) {$\color{purple}{\bullet}$};
				\node at (.4,.6) {$=$};
				\draw[purple] (.8,0) to (.8,1.2);
				\node at (1.2,.6) {$=$};
				\draw[purple] (1.6,0) to (1.6,1.2);
				\draw[magenta][in=-60,out=-90,looseness=1] (1.9,.5) to (1.6,.3);
				\draw[magenta][in=60,out=90,looseness=1] (1.9,.5) to (1.6,.8);
				\node[scale=.8] at (1.6,.3) {$\color{purple}{\bullet}$};
				\node[scale=.8] at (1.6,.8) {$\color{purple}{\bullet}$};
		\end{tikzpicture}} \hspace*{2mm} (separability)
	\end{itemize}
\end{defn}

\begin{defn}
	Given a C*/W*-2-category $\mcal C$, it's \textit{Q-system completion} is the C*/W*-2-category $\textbf{QSys}(\mcal C)$ whose :
	\begin{itemize}
		\item[(1)] 0-cells are Q-systems $(Q,m,i) \in \mcal C_1(b,b)$.
		\item[(2)] 1-cells between Q-systems $P \in \mcal C_1(a,a)$ and $Q \in \mcal C_1(b,b)$ are $Q-P$ bimodules $(X, \lambda_X, \rho_X)$.
		\item[(3)] 2-cells are bimodule intertwiners. Given Q-systems $_aP_a$, $_bQ_b$ and $Q-P$ bimodules $_bX_a$, $_bY_a$, $\textbf{QSys}(\mcal C)_2(_QX_P,\ _QY_P)$ is the set of $f \in \mcal C_2(X,Y)$ such that:
		\[\raisebox{-4mm}{\begin{tikzpicture}
				\draw[purple] (0,0) to (0,1.6);
				\draw[red][in=120,out=90,looseness=1] (-.5,0) to (0,.3);
				\node[scale=.8] at (0,.3) {$\color{purple}{\bullet}$};
				\node[draw,thick,rounded corners,fill=white] at (0,.9) {$f$};
				\node at (.6,.8) {$=$};
				\draw[purple] (1.5,0) to (1.5,1.6);
				\draw[red][in=120,out=90,looseness=1] (1,.8) to (1.5,1.1);
				\node[scale=.8] at (1.5,1.1) {$\color{purple}{\bullet}$};
				\node[draw,thick,rounded corners,fill=white] at (1.5,.6) {$f$};
				\draw[red] (1,.8) to (1,0);
		\end{tikzpicture}} \ \ \ \ \t{and}\ \ \ \
		\raisebox{-4mm}{\begin{tikzpicture}
				\draw[purple] (0,0) to (0,1.6);
				\draw[magenta][in=60,out=90,looseness=1] (.5,0) to (0,.3);
				\node[scale=.8] at (0,.3) {$\color{purple}{\bullet}$};
				\node[draw,thick,rounded corners,fill=white] at (0,.9) {$f$};
				\node at (.6,.8) {$=$};
				\draw[purple] (1.2,0) to (1.2,1.6);
				\draw[magenta][in=60,out=90,looseness=1] (1.7,.8) to (1.2,1.1);
				\node[scale=.8] at (1.2,1.1) {$\color{purple}{\bullet}$};
				\node[draw,thick,rounded corners,fill=white] at (1.2,.6) {$f$};
				\draw[magenta] (1.7,.8) to (1.7,0);
		\end{tikzpicture}} \]
	\end{itemize}
	We refer the reader to \cite{CPJP} for full details that $\textbf{QSys}(\mcal C)$ is a *-2-category.
\end{defn}

Let us recall the description of canonical inclusion functor *-2-functor $\iota_{\mcal C} : \mcal C \hookrightarrow \textbf{QSys}\left(\mcal C\right)$ already given in \cite[Construction 3.24]{CPJP}.
\begin{itemize}
	\item[(I1)] For $c \in \mcal C_0$, we map $c$ to $1_c \in \mcal C_1(c,c)$ with its obvious Q-system structure as the tensor unit of $\mcal C_1(c,c)$.
	
	\item[(I2)] For a 1-cell  $_a X_b \in \mcal C_1(a,b)$, $X$ itself is a unital Frobenius $1_a$-$1_b$ bimodule object, so we map $X$ to itself.
	
	\item[(I3)] For a 2-cell $f \in \mcal C_2(X,Y)$, we see that $f$ is $1_a - 1_b$ bimodular, so we map $f$ to itself. 
\end{itemize}
For further details about the canonical inclusion functor *-2-functor $\iota_{\mcal C} : \mcal C \hookrightarrow \textbf{QSys}\left(\mcal C\right)$ we refer the reader to \cite{CPJP}. 
\begin{defn}
	A 2-category $\mcal C$ is said to be Q-system complete if the canonical inclusion functor *-2-functor $\iota_{\mcal C} : \mcal C \hookrightarrow \textbf{QSys}\left(\mcal C\right)$ is a *-2-equivalence.
\end{defn}
In \cite{CPJP}, Q-system completion of a C*-2-category $\mcal C$  has been characterised in terms of Q-systems which `splits'.

\begin{thm}\cite[Theorem 3.36]{CPJP}.
	A C*/W*-2-category is said to be Q-system complete if and only if every Q-system $Q \in \mcal C_1(b,b)$ splits, that is, there is an object $c \in \mcal C_0$ and a dualizable 1-cell $X \in \mcal C_1(c,b)$ which admits a unitary separable dual $\left(\ol{X}, ev_X, coev_X \right)$ such that $\left(Q,m,i\right)$ is isomorphic to $_b X \us{c}\boxtimes \ol{X}_b$ as Q-systems.
\end{thm}

\subsection{C*-2-category of $ 2 $-functors}\

Suppose $\mcal C, \mcal D$ are C*-2-categories.
\begin{defn}\cite{CP}
	A $ * $-$ 2 $-functor $F = \left(F, F^1,F^2\right) : \mcal C \to \mcal D$ consist of unitary $ 2 $-cells $$F_{X,Y}^2 \in \mcal D_2 \left(F(X) \boxtimes F(Y), F(X \boxtimes Y)\right) \ \t{and} \ F_a^1 \in \mcal D_2 \left(1_{F(a)}, F(1_a)\right)$$ satisfying the following:
	\[\raisebox{-18mm}{\begin{tikzpicture}
			\draw (-.1,.2) to (-.1,1.8);
			\draw (0,.2) to (0,1.8);
			\draw (.1,.2) to (.1,1.8);
			\draw (-.4,-.2) to (-.4,-.8);
			\draw (-.3,-.2) to (-.3,-.8);
			\draw (.5,-.2) to (.5,-2);
			\draw (-.8,-1.4) to (-.8,-2);
			\draw (-.1,-1.4) to (-.1,-2);
			\node[scale=.8] at (.1,2) {$F(X \boxtimes (Y \boxtimes Z))$};
			\node[scale=.8] at (-1,-2.3) {$F(X)$};
			\node[scale=.8] at (-.1,-2.3) {$F(Y)$};
			\node[scale=.8] at (.8,-2.3) {$F(Z)$};
			\node[draw,thick,rounded corners,fill=white,minimum width=20] at (0,0) {$F_{XY,Z}^2$};
			\node[draw,thick,rounded corners,fill=white] at (-.4,-1.2) {$F_{X,Y}^2$};
			\node[draw,thick,rounded corners,fill=white] at (0,1.2) {$F(\alpha^{\mcal C})$};
	\end{tikzpicture}} = \raisebox{-18mm}{\begin{tikzpicture}
			\draw (-.1,.2) to (-.1,.8);
			\draw (0,.2) to (0,.8);
			\draw (.1,.2) to (.1,.8);
			\draw (-.6,-.2) to (-.6,-3);
			\draw (.4,-.2) to (.4,-.8);
			\draw (.5,-.2) to (.5,-3);
			\draw (-.1,-1.4) to (-.1,-3);
			\node[scale=.8] at (.1,1) {$F(X \boxtimes (Y \boxtimes Z))$};
			\node[scale=.8] at (-1,-3.3) {$F(X)$};
			\node[scale=.8] at (-.1,-3.3) {$F(Y)$};
			\node[scale=.8] at (.8,-3.3) {$F(Z)$};
			\node[draw,thick,rounded corners,fill=white,minimum width=20] at (0,0) {$F_{X,YZ}^2$};
			\node[draw,thick,rounded corners,fill=white] at (.3,-1.2) {$F_{Y,Z}^2$};
			\node[draw,thick,rounded corners,fill=white,minimum width=45] at (0,-2.4) {$\alpha^{\mcal D}$};
	\end{tikzpicture}} \ \ \ \ \ \ \ \ \raisebox{-12mm}{\begin{tikzpicture}
	\draw (0,.3) to (0,2);
	\draw (-.4,-.3) to (-.4,-2);
	\draw (.3,-.3) to (.3,-1);
	\draw[densely dotted] (.3,-1.5) to (.3,-2);
	\draw[densely dotted] (.1,.3) to (.1,1.5);
	\node[scale=.8] at (-.9,-1) {$F(X)$};
	\node[scale=.8] at (.8,-.7) {$F(1_a)$};
	\node[scale=.8] at (.8,-1.8) {$1_F(a)$}; 
	\node[draw,thick,rounded corners,fill=white] at (0,0) {$F_{X,1_a}^2$};
	\node[draw,thick,rounded corners,fill=white] at (.3,-1.2) {$F_a^1$};
	\node[draw,thick,rounded corners,fill=white] at (0,1.2) {$F(\rho^a_X)$};
\end{tikzpicture}} = \raisebox{-6mm}{\begin{tikzpicture}
\draw (0,1.5) to (0,-1.5);
\draw[densely dotted] (.4,0) to (.4,-1.5);
\node[draw,thick,rounded corners,fill=white] at (0,0) {$\rho^{F(a)}_{F(X)}$};
\end{tikzpicture}} \ \ \ \ \ \ \ \ \raisebox{-12mm}{\begin{tikzpicture}
\draw (0,.3) to (0,2);
\draw (-.4,-.3) to (-.4,-1.2);
\draw (.3,-.3) to (.3,-2);
\draw[densely dotted] (-.4,-1.5) to (-.4,-2);
\draw[densely dotted] (-.1,.3) to (-.1,1);
\node[scale=.8] at (.8,-1) {$F(X)$};
\node[scale=.8] at (-.9,-.6) {$F(1_b)$};
\node[scale=.8] at (-.9,-1.8) {$1_{F(b)}$}; 
\node[draw,thick,rounded corners,fill=white] at (0,0) {$F_{1_b,X}^2$};
\node[draw,thick,rounded corners,fill=white] at (-.4,-1.2) {$F_b^1$};
\node[draw,thick,rounded corners,fill=white] at (0,1.2) {$F(\lambda^b_X)$};
\end{tikzpicture}} = \raisebox{-6mm}{\begin{tikzpicture}
\draw (0,1.5) to (0,-1.5);
\draw[densely dotted] (-.4,0) to (-.4,-1.5);
\node[draw,thick,rounded corners,fill=white] at (0,0) {$\lambda^{F(b)}_{F(X)}$};
\end{tikzpicture}}  \] for $ 1 $-cells $X \in \mcal C_1(a,b)$, $Y \in \mcal C_1(c,a)$, and $Z \in \mcal C_1(d,c)$
\end{defn}

\begin{defn}
	Suppose $F, G : \mcal C \to \mcal D$ are $ * $-$ 2 $-functors. A $ * $-$ 2 $-transformation $ \phi : F \Rightarrow G$ consists of :
	\begin{itemize}
		\item [(1)] for each $a \in \mcal C_0$, a $ 1 $-cell $\phi_a \in \mcal D_1(F(a), G(a))$, and
		\item [(2)] for each $ 1 $-cell $X \in \mcal C_1(a,b)$, a unitary $ 2 $-cell $\phi_X \in \mcal D_2(\phi_b \boxtimes F(X), G(X) \boxtimes \phi_a)$. We will denote $\phi_X$ by $\raisebox{-8mm}{\begin{tikzpicture}
				\draw[in=-90,out=90,looseness=2] (0,0) to (.5,1);
				\draw[in=-90,out=90,looseness=2] (.5,0) to (0,1);
				\draw[fill=white] (.25,.5) circle(2pt);
				\node at (.6,1.2) {$\phi_a$};
				\node[scale=.7] at (-.2,1.2) {$G(X)$};
				\node[scale=.7] at (.7,-.2) {$F(X)$};
				\node at (-.1,-.2) {$\phi_b$};
		\end{tikzpicture}} $.
	\end{itemize}
This data satisfies the following coherence properties:
$$\raisebox{-16mm}{\begin{tikzpicture}
		\draw[in=-90,out=90,looseness=2] (0,0) to (2,2);
		\draw (2,2) to (2,2.7);
		\draw[in=-90,out=90,looseness=2] (1,0) to (.5,1.7);
		\draw[in=-90,out=90,looseness=2] (1.6,0) to (1.1,1.7);
		\draw[fill=white] (.6,1) circle(2pt);
		\draw[fill=white] (1.2,1) circle(2pt);
		\draw (.8,2.1) to (.8,2.7);
		\draw (.9,2.1) to (.9,2.7);
		\node[draw,thick,rounded corners,fill=white] at (.8,1.8) {$G_{X,Y}^2$};
		\node[scale=.8] at (0,2.5) {$G(X \boxtimes Y)$};
		\node[scale=.8] at (.5,0) {$F(X)$};
		\node[scale=.8] at (2,0) {$F(Y)$};
		\node at (-.3,0) {$\phi_b$};
		\node at (2.3,1.6) {$\phi_c$};	
\end{tikzpicture}} = \raisebox{-16mm}{\begin{tikzpicture}
		\draw[in=-90,out=90,looseness=2] (-.3,.3) to (-.8,1.3);
		\draw[in=-90,out=90,looseness=2] (-.2,.3) to (-.7,1.3);
		\draw[in=-90,out=90,looseness=2] (-.8,.3) to (-.2,1.3);
		\draw (-.8,.3) to (-.8,-1);
		\draw[fill=white] (-.5,.8) circle(2pt);
		\draw (-.3,-.3) to (-.3,-1);
		\draw (.3,-.3) to (.3,-1);
		\node[draw,thick,rounded corners,fill=white] at (0,0) {$F_{X,Y}^2$};
		\node[scale=.8] at (-.3,-1.3) {$F(X)$};
		\node[scale=.8] at (.5,-1.3) {$F(Y)$};
		\node at (-1.1,-.7) {$\phi_b$};
		\node at (.1,1.2) {$\phi_c$}; 
\end{tikzpicture}} \ \ \ \ \ \ \ \ \ \ \ \ \ \ \ \ \raisebox{-12mm}{\begin{tikzpicture}
		\draw[in=-90,out=90,looseness=2] (0,0) to (.5,1);
		\draw[in=-90,out=90,looseness=2] (.5,0) to (0,1);
		\draw[fill=white] (.25,.5) circle(2pt);
		\draw (-.2,1) to (-.2,1.7);
		\draw (.5,1) to (.5,1.7);
		\node[scale=.8,draw,thick,rounded corners,fill=white] at (-.2,1) {$G(f)$};
		\node at (.8,1.2) {$\phi_a$};
		\node[scale=.7] at (-.2,1.9) {$G(Y)$};
		\node[scale=.7] at (.7,-.2) {$F(X)$};
		\node at (-.1,-.2) {$\phi_b$};
\end{tikzpicture}} = \raisebox{-12mm}{\begin{tikzpicture}
		\draw[in=-90,out=90,looseness=2] (0,0) to (.5,1);
		\draw[in=-90,out=90,looseness=2] (.5,0) to (0,1);
		\draw[fill=white] (.25,.5) circle(2pt);
		\draw (.7,-.5) to (.7,-1);
		\draw (0,-1.1) to (0,0);
		\node[scale=.8,draw,thick,rounded corners,fill=white] at (.7,-.3) {$F(f)$};
		\node at (.6,1.2) {$\phi_a$};
		\node[scale=.7] at (-.2,1.2) {$G(Y)$};
		\node[scale=.7] at (.7,-1.2) {$F(X)$};
		\node at (-.1,-1.2) {$\phi_b$};
\end{tikzpicture}} 
\ \ \ \ \ \ \ \ \ \ \ \ \ \ \ \ \ \raisebox{-12mm}{\begin{tikzpicture}
		\draw[in=-90,out=90,looseness=2] (0,0) to (.5,1);
		\draw[densely dotted,in=-90,out=90,looseness=2] (.5,0) to (0,1);
		\draw[fill=white] (.25,.5) circle(2pt);
		\draw (-.2,1) to (-.2,1.7);
		\draw (.5,1) to (.5,1.7);
		\node[scale=.8,draw,thick,rounded corners,fill=white] at (-.1,1) {$G_a^1$};
		\node at (.8,1.2) {$\phi_a$};
		\node[scale=.7] at (-.2,1.9) {$G(1_a)$};
		\node at (-.1,-.2) {$\phi_a$};
\end{tikzpicture}} = \raisebox{-12mm}{\begin{tikzpicture}
		\draw[in=-90,out=90,looseness=2] (0,0) to (.5,1);
		\draw[in=-90,out=90,looseness=2] (.5,0) to (0,1);
		\draw[fill=white] (.25,.5) circle(2pt);
		\draw[densely dotted] (.7,-.5) to (.7,-1);
		\draw (0,-1.1) to (0,0);
		\node[scale=.8,draw,thick,rounded corners,fill=white] at (.7,-.3) {$F_a^1$};
		\node at (.6,1.2) {$\phi_a$};
		\node[scale=.7] at (-.2,1.2) {$G(1_a)$};
		\node at (-.1,-1.2) {$\phi_a$};
\end{tikzpicture}}$$  for $ 1 $-cells $X \in \mcal C_1(a,b)$ and $Y \in \mcal C_1(c,a)$, for every $ 2 $-cell $f \in \mcal C_2 (_bX_a, \,  _bY_a)$
\end{defn}
\begin{defn}\cite{CP}
	Suppose $F,G : \mcal C \to \mcal D$ are $ * $-$ 2 $-functors and $\phi, \psi : F \Rightarrow G$ are $ * $-$ 2 $-transformations. A $ * $-$ 2 $-modification $\eta : \phi \Rrightarrow \psi$ consists of a $ 2 $-cell $\eta_a \in \mcal D_1(\phi_a, \psi_a)$ for each $a \in \mcal C_0$ such that, $\us{a \in \mcal C_0}{\t{sup}}\norm{\eta_a} < \infty$ and 
	\begin{equation}\label{modificationeqns}
		\raisebox{-11mm}{\begin{tikzpicture}
				\draw[in=-90,out=90,looseness=2] (0,0) to (.5,1);
				\draw[in=-90,out=90,looseness=2] (.5,0) to (0,1);
				\draw[fill=white] (.25,.5) circle(2pt);
				\draw (-.1,-.3) to (-.1,-1);
				\draw (.5,0) to (.5,-1);
				\node[draw,thick,rounded corners,fill=white] at (-.1,-.2) {$\eta_b$};
				\node at (.6,1.2) {$\psi_a$};
				\node[scale=.7] at (-.2,1.2) {$G(X)$};
				\node[scale=.7] at (.7,-1.2) {$F(X)$};
				\node at (-.2,.3) {$\psi_b$};
				\node at (-.2,-1.1) {$\phi_b$};
		\end{tikzpicture}} = \raisebox{-11mm}{\begin{tikzpicture}
				\draw[in=-90,out=90,looseness=2] (0,0) to (.5,1);
				\draw[in=-90,out=90,looseness=2] (.5,0) to (0,1);
				\draw[fill=white] (.25,.5) circle(2pt);
				\draw (.6,1.4) to (.6,2);
				\draw (0,1) to (0,2);
				\node[draw,thick,rounded corners,fill=white] at (.6,1.2) {$\eta_a$};
				\node at (.7,.65) {$\phi_a$};
				\node at (.9,1.7) {$\psi_a$};
				\node[scale=.7] at (-.4,1.2) {$G(X)$};
				\node[scale=.7] at (.7,-.2) {$F(X)$};
				\node at (-.1,-.2) {$\phi_b$};
		\end{tikzpicture}}  \ \ \ \ \t{for each} \ \ X \in \mcal C_1(a,b)
	\end{equation}
	
\end{defn}
We now move on to describe the tensor structure in $\textbf{Fun}(\mcal C,\mcal D)$. We will denote the tensor in the $ 2 $-categories $\mcal C, \mcal D$ by $\boxtimes$, while that in $\textbf{Fun}(\mcal C, \mcal D)$ will be denoted by $\otimes$.
\begin{defn}(Horizontal composition in $\textbf{Fun}(\mcal C, \mcal D)$).
	Suppose $F,G,H : \mcal C \to \mcal D$ be $ * $-$ 2 $-functors and $\phi : G \Rightarrow F$, $\psi: H \Rightarrow G$ be $ * $-$ 2 $-transformations. The tensor product $\phi \otimes \psi : H \Rightarrow F$ is defined as follows. For each $a \in \mcal C_0$, $\left( \phi \otimes \psi \right)_a \coloneqq \phi_a \us{G(a)}\boxtimes \psi_a$. For each $ 1 $-cell $X \in \mcal C_1(a,b)$, $\left(\phi \otimes \psi \right)_X \coloneqq \raisebox{-8mm}{\begin{tikzpicture}
			\draw[in=-90,out=90,looseness=2] (.6,0) to (1.1,1);
			\draw[in=-90,out=90,looseness=2] (0,0) to (.5,1);
			\draw[in=-90,out=90,looseness=2] (1.1,0) to (0,1);
			\draw[fill=white] (.8,.5) circle(2pt);
			\draw[fill=white] (.3,.5) circle(2pt);
			\node at (-.2,-.1) {$\phi_b$};
			\node at (.9,-.2) {$\psi_b$};
			\node[scale=.8] at (1.6,0) {$H(X)$};
			\node[scale=.8] at (-.2,1.2) {$F(X)$};
			\node at (1.2,1.25) {$\psi_a$};
			\node at (.6,1.3) {$\phi_a$};
	\end{tikzpicture}} $.\\
Suppose $\phi, \phi' : G \Rightarrow F $ and $\psi, \psi' : H \Rightarrow G$ are $ * $-$ 2 $-transformations and $n : \phi \Rrightarrow \phi'$ and $t : \psi \Rrightarrow \psi'$ are $ * $-$ 2 $-modifications. Their tensor product $n \otimes t : \phi \otimes \psi \Rrightarrow \phi' \otimes \psi'$ is defined as,
\[\left(n \otimes t\right)_a \coloneqq n_a \boxtimes t_a \ \ \t{for each} \ \ a \in \mcal C_0\] 
\end{defn}
We denote vertical composition of $ 2 $-cells in the $ 2 $-categories $\mcal C, \mcal D$ by $\cdot$, while that in $\textbf{Fun}(\mcal C, \mcal D)$ will be denoted by $\circ$.
\begin{defn}(Vertical composition in $\textbf{Fun}(\mcal C, \mcal D)$).
	Suppose $F, G : \mcal C \to \mcal D$ are $ * $-$ 2 $-functors and $\phi, \phi', \phi'' : F \Rightarrow G$ are $ * $-$ 2 $-transformations and $n : \phi \Rrightarrow \phi'$, $n' : \phi' \Rrightarrow \phi''$ are $ * $-$ 2 $-modifications. The vertical composition $n' \circ n : \phi \Rrightarrow \phi'$ is defined by, $\left( n' \circ n \right)_a \coloneqq n'_a \cdot n_a$ for each $a \in \mcal C_0$.
\end{defn}
We refer the reader to \cite{CP}for full details that $\textbf{Fun}(\mcal C, \mcal D)$ is a $ 2 $-category. It has already been proved in \cite{CP} that $\textbf{Fun}(\mcal C, \mcal D)$ is C*/W*-2-category whenever $\mcal C, \mcal D$ are C*/W*-2-categories. We prove below that $\textbf{Fun}(\mcal C, \mcal D)$ is locally orthogonal projection complete whenever $\mcal D$ is locally orthogonal projection complete

\begin{prop}
	Suppose $\mcal D$ is a C*-2-category which is locally orthogonal projection complete. Then $\normalfont\textbf{Fun}(\mcal C, \mcal D)$ is locally orthogonal projection complete
\end{prop}
\begin{proof}
	Suppose $F$ and $G$ be two $ * $-$ 2 $-functors  in $\textbf{Fun}(\mcal C, \mcal D)$, $\phi : F \Rightarrow G$ be a $ * $-$ 2 $-transformation, and $p : \phi \Rrightarrow \phi$ be a $ * $-$ 2 $-modification which is also a projection, that is,
	\[ \forall a \in \mcal C_0, \ \ p_a^2 = p_a^* = p_a \in \mcal D_2(\phi_a,\phi_a) \ \ \t{and} \ \ \raisebox{-11mm}{\begin{tikzpicture}
			\draw[in=-90,out=90,looseness=2] (0,0) to (.5,1);
			\draw[in=-90,out=90,looseness=2] (.5,0) to (0,1);
			\draw[fill=white] (.25,.5) circle(2pt);
			\draw (-.1,-.3) to (-.1,-1);
			\draw (.5,0) to (.5,-1);
			\node[draw,thick,rounded corners,fill=white] at (-.1,-.2) {$p_b$};
			\node at (.6,1.2) {$\phi_a$};
			\node[scale=.7] at (-.2,1.2) {$G(X)$};
			\node[scale=.7] at (.7,-1.2) {$F(X)$};
			\node at (-.2,.3) {$\phi_b$};
			\node at (-.2,-1.1) {$\phi_b$};
	\end{tikzpicture}} = \raisebox{-11mm}{\begin{tikzpicture}
	\draw[in=-90,out=90,looseness=2] (0,0) to (.5,1);
	\draw[in=-90,out=90,looseness=2] (.5,0) to (0,1);
	\draw[fill=white] (.25,.5) circle(2pt);
	\draw (.6,1.4) to (.6,2);
	\draw (0,1) to (0,2);
	\node[draw,thick,rounded corners,fill=white] at (.6,1.2) {$p_a$};
	\node at (.7,.65) {$\phi_a$};
	\node at (.9,1.7) {$\phi_a$};
	\node[scale=.7] at (-.4,1.2) {$G(X)$};
	\node[scale=.7] at (.7,-.2) {$F(X)$};
	\node at (-.1,-.2) {$\phi_b$};
\end{tikzpicture}}  \ \ \forall \ X \in \mcal C_1(a,b) \]
	
	So, $p_a \in \mcal D_2(\phi_a, \phi_a)$ is a projection. Since, $\mcal D$ is locally orthogonal projection complete $\exists \ x_a \in \mcal D_1(F(a), G(a))$ and  $i_a \in \mcal D_2(x_a,\phi_a)$ such that,  
	\begin{equation}\label{idempotent}
		\ \ i_a^* i_a = 1_{x_a} \ \ \t{and} \ \ i_a i_a^* = p^{(a)} \ \ \forall a \in \mcal C_0 .  
	\end{equation}
	For a $ 1 $-cell, $X \in \mcal C_1(a,b)$, let $x_X \coloneqq 
	\raisebox{-6mm}{
		\begin{tikzpicture}
			\draw[in=-90,out=90] (0,1) to (0.5,2);
			\draw[white,line width=1mm,in=-90,out=90] (0.5,1) to (0,2);
			\draw[in=-90,out=90] (0.5,1) to (0,2);
			\draw[fill=white] (.25,1.5) circle(2pt);
			\node[scale=.8] at (1,1) {$F(X)$};
			\node at (-.3,1) {$x_b$};
			\node[scale=.8] at (-.5,1.9) {$G(X)$};
			\node at (.8,1.9) {$x_a$};
	\end{tikzpicture}} = 
	\raisebox{-14mm}{
		\begin{tikzpicture}
			\draw[in=-90,out=90] (0,.2) to (0,1.2);
			\draw[in=-90,out=90] (0,1.2) to (0.6,2);
			\draw (.6,2) to (.6,3);
			\draw[in=-90,out=90] (.6,.2) to (0.6,1.2);
			\draw[white,line width=1mm,in=-90,out=90] (0.6,1.2) to (0,2);
			\draw[in=-90,out=90] (0.6,1.2) to (0,2);
			\draw (0,2) to (0,3);
			\draw[fill=white] (.3,1.6) circle(2pt);
			\node[draw,thick,rounded corners, fill=white] at (0,0.8) {$i_b$};
			\node[draw,thick,rounded corners, fill=white] at (.6,2.3) {$i_a^*$};
			\node[scale=.8,right] at (.56,0.2) {$ F(X) $};
			\node[right] at (.5,1.65) {$\phi_a$};
			\node[left] at (0.1,1.4) {$\phi_b$}; 
			\node[left] at (0.1,.2) {$x_b$}; 
			\node[scale=.8,left] at (0,2.9) {$G(X)$};
			\node[right] at (.6,2.9) {$x_a$};
	\end{tikzpicture}}
	$
	
	$x_X x_X^* = \raisebox{-20mm}{
		\begin{tikzpicture}
			\draw[in=-90,out=90] (0,.2) to (0,1.2);
			\draw[in=-90,out=90] (0,1.2) to (0.6,2);
			\draw (.6,2) to (.6,3);
			\draw (0,3.5) to (0,4.4);
			\draw[in=-90,out=90] (.6,3) to (0,3.5);
			\draw[in=-90,out=90] (.6,.2) to (0.6,1.2);
			\draw[white,line width=1mm,in=-90,out=90] (0.6,1.2) to (0,2);
			\draw[in=-90,out=90] (0.6,1.2) to (0,2);
			\draw[fill=white] (.3,1.6) circle(2pt);
			\draw[white,line width=1mm,in=-90,out=90] (0,3) to (.6,3.5);
			\draw[in=-90,out=90] (0,3) to (.6,3.5);
			\draw[fill=white] (.3,3.25) circle(2pt);
			\draw (.6,4) to (.6,4.4);
			\draw (0,2) to (0,3);
			\node[draw,thick,rounded corners, fill=white] at (0.6,0.8) {$i_a$};
			\node[draw,thick,rounded corners, fill=white] at (0,2.3) {$p_b$};
			\node[draw,thick,rounded corners, fill=white] at (.75,3.7) {$i_a^*$};
			\node[right] at (.56,0.2) {$ x_a $};
			\node[right] at (.56,1.35) {$\phi_a $};
			\node[left] at (0.2,1.8) {$\phi_b$}; 
			\node[scale=.8,left] at (0.1,.2) {$G(X)$}; 
			\node[left] at (0,2.9) {$\phi_b$};
			\node[scale=.8,right] at (.6,2.9) {$F(X)$};
			\node[right] at (.6,4.2) {$x_a$};
			\node[scale=.8,left] at (0.1,4.2) {$G(X)$};
	\end{tikzpicture}} = 
	\raisebox{-10mm}{
		\begin{tikzpicture}
			\draw (0,0) to (0,2.4);
			\draw (.6,0) to (.6,2.4);
			\node[scale=.8] at (-.1,-.2) {$G(X)$};
			\node at (.9,-.2) {$x_a$};
	\end{tikzpicture}}$.	
	The last equality follows from the fact that $p : \phi \Rrightarrow \phi$ is a $ * $-$ 2 $-modification and from \Cref{idempotent}. Similarly, it follows that, $x_X^* x_X = 1_{x_b \boxtimes F(X)}$ 
		
	Indeed, $x : F \Rightarrow G$ is a $ * $-$ 2 $-transformation. Now, the proposition will follow from the following claim\\
	\textit{Claim} : $i : x \Rrightarrow \phi$ is $ * $-$ 2 $-modification.\\
	\textit{Proof of the claim} : Since each $i_a$ is an isometry, so $\us{a \in \mcal C_0}{\t{sup}} \norm{i_a} = 1$. Also, $\raisebox{-11mm}{\begin{tikzpicture}
			\draw[in=-90,out=90,looseness=2] (0,0) to (.5,1);
			\draw[in=-90,out=90,looseness=2] (.5,0) to (0,1);
			\draw[fill=white] (.25,.5) circle(2pt);
			\draw (-.1,-.3) to (-.1,-1);
			\draw (.5,0) to (.5,-1);
			\node[draw,thick,rounded corners,fill=white] at (-.1,-.2) {$i_b$};
			\node at (.6,1.2) {$\phi_a$};
			\node[scale=.7] at (-.2,1.2) {$G(X)$};
			\node[scale=.7] at (.7,-1.2) {$F(X)$};
			\node at (-.2,.3) {$\phi_b$};
			\node at (-.2,-1.1) {$x_b$};
	\end{tikzpicture}} = \raisebox{-11mm}{\begin{tikzpicture}
	\draw[in=-90,out=90,looseness=2] (0,0) to (.5,1);
	\draw[in=-90,out=90,looseness=2] (.5,0) to (0,1);
	\draw[fill=white] (.25,.5) circle(2pt);
	\draw (-.1,-.3) to (-.1,-1);
	\draw (.5,0) to (.5,-1);
	\node[draw,thick,rounded corners,fill=white] at (-.1,-.2) {$p_b i_b$};
	\node at (.6,1.2) {$\phi_a$};
	\node[scale=.7] at (-.2,1.2) {$G(X)$};
	\node[scale=.7] at (.7,-1.2) {$F(X)$};
	\node at (-.2,.3) {$\phi_b$};
	\node at (-.2,-1.1) {$x_b$};
\end{tikzpicture}} =  \raisebox{-14mm}{
\begin{tikzpicture}
\draw[in=-90,out=90] (0,.2) to (0,1.2);
\draw[in=-90,out=90] (0,1.2) to (0.6,2);
\draw (.6,2) to (.6,3);
\draw[in=-90,out=90] (.6,.2) to (0.6,1.2);
\draw[white,line width=1mm,in=-90,out=90] (0.6,1.2) to (0,2);
\draw[in=-90,out=90] (0.6,1.2) to (0,2);
\draw (0,2) to (0,3);
\draw[fill=white] (.3,1.6) circle(2pt);
\node[draw,thick,rounded corners, fill=white] at (0,0.8) {$i_b$};
\node[draw,thick,rounded corners, fill=white] at (.6,2.3) {$p_a$};
\node[scale=.8,right] at (.56,0.2) {$ F(X) $};
\node[right] at (.5,1.65) {$\phi_a$};
\node[left] at (0.1,1.4) {$\phi_b$}; 
\node[left] at (0.1,.2) {$x_b$}; 
\node[scale=.8,left] at (0,2.9) {$G(X)$};
\node[right] at (.6,2.9) {$x_a$};
\end{tikzpicture}} = \raisebox{-11mm}{\begin{tikzpicture}
\draw[in=-90,out=90,looseness=2] (0,0) to (.5,1);
\draw[in=-90,out=90,looseness=2] (.5,0) to (0,1);
\draw[fill=white] (.25,.5) circle(2pt);
\draw (.6,1.4) to (.6,2);
\draw (0,1) to (0,2);
\node[draw,thick,rounded corners,fill=white] at (.6,1.2) {$i_a$};
\node at (.7,.65) {$x_a$};
\node at (.9,1.7) {$\phi_a$};
\node[scale=.7] at (-.4,1.2) {$G(X)$};
\node[scale=.7] at (.7,-.2) {$F(X)$};
\node at (-.1,-.2) {$x_b$};
\end{tikzpicture}} $. The first equality follows from \Cref{idempotent}. Now $p : \phi \Rrightarrow \phi$ being a $ * $-$ 2 $-modification reveals the second equality. The third equality again follows from \Cref{idempotent}. This completes the proof of the claim \\
Thus, $\textbf{Fun}(\mcal C,\mcal D)$ is locally orthogonal projection complete.
\end{proof}

	\section{Construction of $0$-cell}\label{2functor}
	For the rest of the paper, we assume $\mcal C$ and $\mcal D$ are strict C*-2-categories. In this section, we construct a suitable $ 2 $-functor $G : \mcal C \to \mcal D$.
	 
	\vspace*{2mm}
	
	Let $\left(\psi_\bullet, m_\bullet , i_\bullet \right)$ be a Q-system in $\t{End}_{\textbf{Fun}(\mcal C,\mcal D)}(F)$ for some $*$-2-functor $F : \mcal C \to \mcal D$ where, $m : \psi \otimes \psi \Rrightarrow \psi$ is the multiplication map and $i : 1_{F} \Rrightarrow \psi$ is the unit map. This means that, for every $a \in \mcal C_0$, we have a Q-system $\left(\psi_a,m_a,i_a\right)$ in $\t{End}_{\mcal D_1} (Fa)$ and $m_a$ and $i_a$ satisfy \Cref{modificationeqns}. We will use these data to construct our suitable $ * $-$ 2 $-functor $G : \mcal C \to \mcal D$.
	
	We will use the following notations:
	 \[m_a = \raisebox{-6mm}{
	 	\begin{tikzpicture}
	 		\draw[red,in=90,out=90,looseness=2] (-0.5,0.5) to (-1.5,0.5);
	 		\node at (-1,1.1) {${\color{red}\bullet}$};
	 		\draw[red] (-1,1.1) to (-1,1.6);
	 		\node[left] at (-1,1.4) {$\psi_a$};
	 		\node[left] at (-1.6,0.5) {$\psi_a$};
	 		\node[right] at (-.5,.5) {$\psi_a$};
	 \end{tikzpicture}} \ \ \ \ \ \ i_a = \raisebox{-6mm}{
	 	\begin{tikzpicture}
	 		\draw [red] (-0.8,-.6) to (-.8,.6);
	 		\node at (-.8,-.6) {${\color{red}\bullet}$};
	 		\node[left] at (-.8,0) {$\psi_a$};
	 \end{tikzpicture}} \ \ \ \ \ \ m_a^* = \raisebox{-6mm}{
	 	\begin{tikzpicture}
	 		\draw[red,in=-90,out=-90,looseness=2] (-0.5,0.5) to (-1.5,0.5);
	 		\node at (-1,-.1) {${\color{red}\bullet}$};
	 		\draw[red] (-1,-.1) to (-1,-.6);
	 		\node[left] at (-1,-.4) {$\psi_a$};
	 		\node[left] at (-1.6,0.5) {$\psi_a$};
	 		\node[right] at (-.5,.5) {$\psi_a$};
	 \end{tikzpicture}} \ \ \ \ \ \ i_a^* = \raisebox{-6mm}{
	 	\begin{tikzpicture}
	 		\draw [red] (-0.8,-.6) to (-.8,.6);
	 		\node at (-.8,.6) {${\color{red}\bullet}$};
	 		\node[left] at (-.8,0) {$\psi_a$};
	 \end{tikzpicture}} \ \ \ \ \t{for each} \ \ a \in \mcal C_0 \]
	
	\subsection{$G$ on $ 0 $-cells.}\
	
	Since $\mcal D$ is Q-system complete, by \cite[Theorem 3.36]{CPJP}, we have an isomorphism of Q-systems $\psi_a$ and $\ol{x}_a \us{c_a}\boxtimes x_a$, where $c_a \in \mcal D_0,\  x_a \in \mcal D_1(Fa,c_a)$ and $\raisebox{-6mm}{\begin{tikzpicture}
	\draw[in=90,out=90,looseness=2,->] (0,0) to (1,0);
	\draw[in=-90,out=-90,looseness=2,<-] (0,0) to (1,0);
	\node at (-.4,-.1) {$x_a$};
	\node at (1.4,-.1) {$\ol{x}_a$};
	\end{tikzpicture}} \ = \ \t{id}_{1_{c_a}} \ \ $. For $a \in \mcal C_0$, define $G(a) \coloneqq c_a \in \mcal D_0$.

\vspace*{2mm}

We move on to define $G$ on $1$-cells and $2$-cells of $\mcal C$.

\subsection{$G$ on $ 1 $-cells}\

Let $_b X_a \in \mcal C_1(a,b)$ and $\gamma^{(a)} \in \mcal D_2 (\ol{x}_a \us{c_a} \boxtimes x_a , \psi_a)$ be the unitary intertwining the algebra maps, for each $a \in \mcal C_0$. We denote \[\gamma^{(a)} \ = \ \raisebox{-12mm}{
	\begin{tikzpicture}
		\draw[red] (0,.3) to (0,1);
		\draw[->] (-.2,-.3) to (-.2,-1);
		\draw[<-] (.2,-.35) to (.2,-1);
		\node[draw,thick,rounded corners,fill=white] at (0,0) {$\gamma^{(a)}$};
		\node at (-.3,-1.2) {$\ol{x}_a$};
		\node at (.4,-1.2) {$x_a$};
		\node at (.3,.8) {$\psi_a$};
\end{tikzpicture}} \ , \ \psi_X \ = \ \raisebox{-8mm}{\begin{tikzpicture}
		\draw[red,in=-90,out=90,looseness=2] (0,0) to (.5,1);
		\draw[in=-90,out=90,looseness=2] (.5,0) to (0,1);
		\draw[fill=white] (.25,.5) circle(2pt);
		\node at (.6,1.2) {$\psi_a$};
		\node[scale=.7] at (-.2,1.2) {$F(X)$};
		\node[scale=.7] at (.7,-.2) {$F(X)$};
		\node at (-.1,-.2) {$\psi_b$};
\end{tikzpicture}}  \ \t{and} \ \psi_X^* \ = \ \raisebox{-8mm}{\begin{tikzpicture}
\draw[in=-90,out=90,looseness=2] (0,0) to (.5,1);
\draw[red,in=-90,out=90,looseness=2] (.5,0) to (0,1);
\draw[fill=white] (.25,.5) circle(2pt);
\node at (.6,-.2) {$\psi_a$};
\node[scale=.7] at (-.2,-.2) {$F(X)$};
\node[scale=.7] at (.7,1.2) {$F(X)$};
\node at (-.1,1.2) {$\psi_b$};
\end{tikzpicture}}\] 

We have, for each $a \in \mcal C_0$ 
	\begin{equation}\label{gammaisoeqn}
	\raisebox{-8mm}{
	\begin{tikzpicture}
		\draw[->] (-.2,-.3) to (-.2,-1);
		\draw[<-] (.2,-.35) to (.2,-1);
		\draw[->] (.8,-.3) to (.8,-1);
		\draw[<-] (1.2,-.35) to (1.2,-1);
		\draw[red,in=90,out=90,looseness=2] (0,.3) to (1,.3);
		\node at (.5,.88) {$\red{\bullet}$};
		\draw[red] (.5,.88) to (.5,1.4);
		\draw[in=90,out=90,looseness=2,->] (2.7,-1) to (3.2,-1);
		\node at (2,0) {$=$};
		\draw[->] (2.5,-.3) to (2.5,-1);
		\draw[<-] (3.4,-.35) to (3.4,-1);
		\draw[red] (3,.3) to (3,1.2);
		\node[draw,thick,rounded corners,fill=white] at (0,0) {$\gamma^{(a)}$};
		\node[draw,thick,rounded corners,fill=white] at (1,0) {$\gamma^{(a)}$};
		\node[draw,thick,rounded corners,fill=white,minimum width=35] at (3,0) {$\gamma^{(a)}$};
\end{tikzpicture}} \ \t{and} \ 
\raisebox{-10mm}{
	\begin{tikzpicture}
		\draw[red] (0,.3) to (0,1);
		\draw[->] (-.2,-.3) to (-.2,-1);
		\draw[<-] (.2,-.3) to (.2,-1);
		\draw[in=-90,out=-90,looseness=2] (-.2,-1) to (.2,-1);
		\node[draw,thick,rounded corners,fill=white] at (0,0) {$\gamma^{(a)}$};
\end{tikzpicture}} \ = \ 
\raisebox{-6mm}{
	\begin{tikzpicture}
		\draw[red] (0,-.4) to (0,1);
		\node at (0,-.4) {$\red{\bullet}$};
\end{tikzpicture}}
\end{equation}

It is straightforward and an enjoyable exercise in graphical calculus to prove the following lemma.
\begin{lem}\label{Qsysgammaklem}
	For each $a \in \mcal C_0$, we have the following:
	\begin{itemize}
		\item[(i)] $\raisebox{-10mm}{\begin{tikzpicture}
				\draw[red,in=90,out=90,looseness=2] (0,.3) to (-1,.3);
				\draw[red] (-1,.3) to (-1,-.8);
				\draw[red] (-.5,.9) to (-.5,1.5);
				\draw[<-] (.7,-.35) to (.7,1.5);
				\draw[->] (1.1,-.35) to (1.1,1.5);
				\draw[in=-90,out=-90,looseness=2] (.2,-.35) to (.7,-.35);
				\draw[in=-90,out=-90,looseness=2] (-.2,-.35) to (1.1,-.35);
				\node[draw,thick,rounded corners,fill=white] at (0,0) {$\gamma^{(a)}$};
				\node at (-.5,.9) {$\red{\bullet}$};
				\node at (-.5,1.5) {$\red{\bullet}$};
		\end{tikzpicture}} \ = \ \raisebox{-4mm}{\begin{tikzpicture}
				\draw[red] (0,-.3) to (0,-1);
				\draw[<-] (-.3,.35) to (-.3,1);
				\draw[->] (.3,.35) to (.3,1);
				\node[draw,thick,rounded corners,fill=white] at (0,0) {${\gamma^{(a)}}^*$};
		\end{tikzpicture}} \ = \ \raisebox{-10mm}{\begin{tikzpicture}
				\draw[red,in=90,out=90,looseness=2] (0,.3) to (1,.3);
				\draw[red] (1,.3) to (1,-.8);
				\draw[red] (.5,.9) to (.5,1.5);
				\draw[->] (-.7,-.35) to (-.7,1.5);
				\draw[<-] (-1.1,-.35) to (-1.1,1.5);
				\draw[in=-90,out=-90,looseness=2] (.2,-.35) to (-1.1,-.35);
				\draw[in=-90,out=-90,looseness=2] (-.2,-.35) to (-.7,-.35);
				\node[draw,thick,rounded corners,fill=white] at (0,0) {$\gamma^{(a)}$};
				\node at (.5,.9) {$\red{\bullet}$};
				\node at (.5,1.5) {$\red{\bullet}$};
		\end{tikzpicture}}  $  
		
		\item[(ii)] $\raisebox{-6mm}{\begin{tikzpicture}
				\draw[red] (0,.3) to (0,1);
				\draw[->] (-.3,-.35) to (-.3,-1);
				\draw[<-] (.3,-.35) to (.3,-1);
				\node[draw,thick,rounded corners,fill=white] at (0,0) {$\gamma^{(a)}$};
				\node at (0,1) {$\red{\bullet}$};
		\end{tikzpicture}} \ = \ \begin{tikzpicture}
			\draw[in=90,out=90,looseness=2,<-] (0,0) to (1,0); 
		\end{tikzpicture}$
	\end{itemize}
\end{lem}

\begin{cor}\label{Qsysgammakcor}
	For each $a \in \mcal C_0$, $\raisebox{-6mm}{\begin{tikzpicture}
			\draw[red] (0,0) to (0,.7);
			\draw[->] (-.2,-.3) to (-.2,-1);
			\draw[<-] (.2,-.35) to (.2,-1);
			\draw[red,in=-90,out=-90,looseness=2] (-.4,1.2) to (.4,1.2);
			\node[draw,thick,rounded corners,fill=white] at (0,0) {$\gamma^{(a)}$};
			\node at (0,.75) {$\red{\bullet}$};
	\end{tikzpicture}} \ = \ \raisebox{-6mm}{\begin{tikzpicture}
			\draw[red] (0,0) to (0,1);
			\draw[red] (1.5,0) to (1.5,1);
			\draw[->] (-.2,-.35) to (-.2,-1);
			\draw[<-] (1.7,-.35) to (1.7,-1);
			\draw[in=-90,out=-90,looseness=2,<-] (.2,-.35) to (1.3,-.35);
			\node[draw,thick,rounded corners,fill=white] at (0,0) {$\gamma^{(a)}$};
			\node[draw,thick,rounded corners,fill=white] at (1.5,0) {$\gamma^{(a)}$};
	\end{tikzpicture}} \ \ \ \ $. 
\end{cor}

\comments{\begin{rem}\label{Gammaint}
	\red{From \Cref{Qsysgammaklem} and \Cref{Qsysgammakcor}, it follows that  for each $k \geq l$, $\gamma^{(k)}$ is both a $*$-homomorphism and a $*$-cohomorphism in the sense of [defn. 2.1 \red{Ref}]}. 
\end{rem}}

Define $p_X \coloneqq \raisebox{-11mm}{\begin{tikzpicture}
\draw[->] (-.1,-.8) to (-.1,-1.5);
\draw[in=-90,out=-90,looseness=2] (.4,-.8) to (1.1,-.8);
\draw[<-] (1.1,-.8) to (1.1,1.5);
\draw[->] (-.6,.9) to (-.6,1.5);
\draw[in=90,out=90,looseness=2] (-1,.9) to (-1.7,.9);
\draw[<-] (-1.7,.9) to (-1.7,-1.5);
\draw (-.5,-.5) to (-.5,-1.5);
\draw (0,.5) to (0,1.5);
\draw[in=-90,out=90,looseness=2] (-.5,-.5) to (0,.5);
\draw[red,in=-90,out=90,looseness=2] (0,-.5 ) to (-.5,.5);
\draw[fill=white] (-.25,0) circle(2pt);
\node[draw,thick,fill=white,rounded corners] at (0.2,-.6) {$\gamma^{(a)}$};
\node[draw,thick,fill=white,rounded corners] at (-.8,.6) {${\gamma^{(b)}}^*$};
\end{tikzpicture}} \in {\normalfont\t{End}}_{\mcal D_1}\left(x_b \us{F_b} \boxtimes F(X) \us{F_a} \boxtimes \bar{x}_a \right)$ 
	  
\begin{lem}\label{proj1cell}
	For each $a, b \in \mcal C_0$ and $X \in \mcal C_1(a,b)$, $p_X$ is a projection in ${\normalfont\t{End}}_{\mcal D_1}\left(x_b \us{F_b} \boxtimes F(X) \us{F_a} \boxtimes \bar{x}_a \right)$ 
\end{lem}
\begin{proof}
We have $p_X^2 \ = \ \raisebox{-8mm}{\begin{tikzpicture}
		\draw[in=-90,out=90,looseness=2] (0,0) to (.5,1);
		\draw[in=-90,out=90,looseness=2] (.5,1) to (1,2);
		\draw (1,2) to (1,2.8);
		\draw (0,0) to (0,-.7);
		\draw[white,line width=1mm,in=-90,out=90,looseness=2] (.5,0) to (0,1);
		\draw[red,in=-90,out=90,looseness=2] (.5,0) to (0,1);
		\draw[fill=white] (.25,.5) circle(2pt);
		\draw[white,line width=1mm,in=-90,out=90,looseness=2] (1,1) to (.5,2);
		\draw[red,in=-90,out=90,looseness=2] (1,1) to (.5,2);
		\draw[fill=white] (.75,1.5) circle(2pt);
		\draw[->] (.4,0) to (.4,-.7);
		\draw[in=-90,out=-90,looseness=2] (.7,-.2) to (1.2,-.2);
		\draw[<-] (1.2,-.2) to (1.2,1);
		\draw[<-] (1.6,1) to (1.6,-.7);
		\draw[in=90,out=90,looseness=2] (.3,2.3) to (-.2,2.3);
		\draw[in=90,out=90,looseness=2] (-.6,1.3) to (-1.2,1.3);
		\draw[->] (-.2,1.3) to (-.2,2.3);
		\draw[->] (-1.2,-.7) to (-1.2,1.3);
		\draw[->] (.6,2.3) to (.6,2.8);
		\node[scale=.8,draw,thick,rounded corners,fill=white] at (.56,0) {$\gamma^{(a)}$};
		\node[scale=.8,draw,thick,rounded corners,fill=white,minimum width=45] at (1.3,1) {$\gamma^{(a)}$};
		\node[scale=.8,draw,thick,rounded corners,fill=white,minimum width=45] at (-.3,1.1) {${\gamma^{(b)}}^*$};
		\node[scale=.8,draw,thick,rounded corners,fill=white] at (.4,2.1) {${\gamma^{(b)}}^*$};
\end{tikzpicture}} \ = \ \raisebox{-10mm}{\begin{tikzpicture}
		\draw[in=-90,out=90,looseness=2] (0,0) to (1.5,2);
		\draw (0,0) to (0,-.7);
		\draw (1.5,2) to (1.5,2.8);
		\draw[white,line width=1mm,in=-90,out=90,looseness=2] (1,.8) to (.8,1.2);
		\draw[white,line width=1mm,in=-90,out=90,looseness=2] (.5,.8) to (.3,1.2);
		\draw[red,in=-90,out=90,looseness=2] (1,.8) to (.8,1.2);
		\draw[fill=white] (.9,1) circle(2pt);
		\draw[red,in=-90,out=90,looseness=2] (.5,.8) to (.3,1.2);
		\draw[fill=white] (.45,.95) circle(2pt);
		\draw[red,in=-90,out=-90,looseness=2] (1,.8) to (.5,.8);
		\draw[red,in=90,out=90,looseness=2] (.8,1.2) to (.3,1.2);
		\node[scale=.8] at (.8,.5) {$\red{\bullet}$};
		\node[scale=.8] at (.5,1.5) {$\red{\bullet}$};
		\draw[red] (.8,.5) to (.8,.2);
		\draw[red] (.5,1.5) to (.5,1.8);
		\draw[->] (.7,-.2) to (.7,-.7);
		\draw[in=-90,out=-90,looseness=2] (1.1,-.2) to (1.8,-.2);
		\draw[<-] (1.8,-.2) to (1.8,2.8);
		\draw[in=90,out=90,looseness=2] (-.6,2.3) to (.1,2.3); 
		\draw[->] (.6,2.3) to (.6,2.8);
		\draw[<-] (-.6,2.3) to (-.6,-.7);
		\node[scale=.8,draw,thick,rounded corners,fill=white] at (.9,0) {$\gamma^{(a)}$};
		\node[scale=.8,draw,thick,rounded corners,fill=white] at (.4,2.1) {${\gamma^{(b)}}^*$};
\end{tikzpicture}} \ = \ p_k$. The second equality follows from \Cref{Qsysgammakcor} and the last equality follows from separability condition and \Cref{modificationeqns} for $m_a$'s. 

$p_X^* = \raisebox{-8mm}{\begin{tikzpicture}
		\draw[in=-90,out=90,looseness=2] (.5,0) to (0,1);
		\draw (0,1) to (0,1.7);
		\draw (.5,0) to (.5,-.8);
		\draw[white,line width=1mm,in=-90,out=90,looseness=2] (0,0) to (.5,1);
		\draw[red,in=-90,out=90,looseness=2] (0,0) to (.5,1);
		\draw[fill=white] (.25,.5) circle(2pt);
		\draw[<-] (.4,1.2) to (.4,1.7);
		\draw[in=90,out=90,looseness=2,->] (.7,1.3) to (1.2,1.3);
		\draw[<-] (0,-.3) to (0,-.8);
		\draw[in=-90,out=-90,looseness=2,->] (-.6,-.4) to (-1.1,-.4);
		\draw (-1.1,-.4) to (-1.1,1.7);
		\draw (1.2,1.3) to (1.2,-.8);
		\node[scale=.8,draw,thick,rounded corners,fill=white] at (-.3,-.1) {$\gamma^{(b)}$};
		\node[scale=.8,draw,thick,rounded corners,fill=white] at (.6,1) {${\gamma^{(a)}}^*$};
\end{tikzpicture}} \ = \ \raisebox{-14mm}{\begin{tikzpicture}
		\draw[in=-90,out=90,looseness=2] (.5,0) to (0,1);
		\draw (0,1) to (0,1.7);
		\draw (.5,0) to (.5,-.8);
		\draw[white,line width=1mm,in=-90,out=90,looseness=2] (0,0) to (.5,1);
		\draw[red,in=-90,out=90,looseness=2] (0,0) to (.5,1);
		\draw[fill=white] (.25,.5) circle(2pt);
		\draw[<-] (.4,1.2) to (.4,1.7);
		\draw[in=90,out=90,looseness=2,->] (.7,1.3) to (1.2,1.3);
		\draw[<-] (-1.5,-.3) to (-1.5,1.3);
		\draw[in=-90,out=-90,looseness=2,->] (-.6,-.4) to (-1.1,-.4);
		\draw (-1.1,-.4) to (-1.1,1.7);
		\draw (1.2,1.3) to (1.2,-.8);
		\draw[in=-90,out=-90,looseness=2] (0,-.3) to (-1.5,-.3);
		\draw[in=90,out=90,looseness=2,->] (-1.5,1.3) to (-2,1.3);
		\draw (-2,1.3) to (-2,-1);
		\node[scale=.8,draw,thick,rounded corners,fill=white] at (-.3,-.1) {$\gamma^{(b)}$};
		\node[scale=.8,draw,thick,rounded corners,fill=white] at (.6,1) {${\gamma^{(a)}}^*$};
\end{tikzpicture}} \ = \ \raisebox{-8mm}{\begin{tikzpicture}
		\draw[in=-90,out=90,looseness=2] (.5,0) to (0,1);
		\draw (0,1) to (0,1.7);
		\draw (.5,0) to (.5,-.8);
		\draw[white,line width=1mm,in=-90,out=90,looseness=2] (0,0) to (.5,1);
		\draw[red,in=-90,out=90,looseness=2] (0,0) to (.5,1);
		\draw[fill=white] (.25,.5) circle(2pt);
		\draw[red,in=-90,out=-90,looseness=2] (0,0) to (-.5,0);
		\node[scale=.8] at (-.25,-.3) {$\red{\bullet}$};
		\draw[red] (-.25,-.3) to (-.25,-.7);
		\node[scale=.8] at (-.25,-.7) {$\red{\bullet}$};
		\draw[in=90,out=90,looseness=2] (.4,1.3) to (1.6,1.3);
		\draw[in=90,out=90,looseness=2,->] (.7,1.3) to (1.2,1.3);
		\draw[->] (-.6,.6) to (-.6,1.7);
		\draw[in=90,out=90,looseness=2] (-1.1,.6) to (-1.6,.6);
		\draw[<-] (-1.6,.6) to (-1.6,-.8);
		\draw (1.2,1.3) to (1.2,-.8);
		\draw[<-] (1.6,1.3) to (1.6,-.4);
		\draw[in=-90,out=-90,looseness=2] (1.6,-.4) to (2.1,-.4);
		\draw[<-] (2.1,-.4) to (2.1,1.7);
		\node[scale=.8,draw,thick,rounded corners,fill=white] at (-.8,.3) {${\gamma^{(b)}}^*$};
		\node[scale=.8,draw,thick,rounded corners,fill=white] at (.6,1) {${\gamma^{(a)}}^*$};
\end{tikzpicture}} \ = \ \raisebox{-8mm}{\begin{tikzpicture}
		\draw[in=-90,out=90,looseness=2] (-.5,-.5) to (0,.5);
		\draw[white,line width=1mm,in=-90,out=90,looseness=2] (0,-.5 ) to (-.5,.5);
		\draw (-.5,-.5) to(-.5,-1.5);
		\draw (0,.5) to (0,1.5);
		\draw[red,in=-90,out=90,looseness=2] (0,-.5 ) to (-.5,.5);
		\draw[fill=white] (-.25,0) circle(2pt);
		\draw[red,in=-90,out=-90,looseness=2] (0,-.5 ) to (.5,-.5);
		\draw[red,in=90,out=90,looseness=2] (.5,-.5 ) to (1,-.5);
		\node[scale=.8] at (.26,-.8) {$\red{\bullet}$};
		\node[scale=.8] at (.76,-.2) {$\red{\bullet}$};
		\draw[red] (.26,-.8) to (.26,-1.2);
		\draw[red] (.76,-.2) to (.76,.2);
		\node[scale=.8] at (.26,-1.2) {$\red{\bullet}$};
		\node[scale=.8] at (.76,.2) {$\red{\bullet}$};
		\draw[->] (-.8,1) to (-.8,1.5);
		\draw[in=90,out=90,looseness=2] (-1.2,1) to (-1.7,1);
		\draw[<-] (-1.7,1) to (-1.7,-1.5);
		\draw[->] (1,-1.1) to (1,-1.5);
		\draw[in=-90,out=-90,looseness=2] (1.4,-1.1) to (1.9,-1.1);
		\draw[<-] (1.9,-1.1) to (1.9,1.5);
		\node[scale=.8,draw,thick,rounded corners,fill=white] at (-.85,.7) {${\gamma^{(b)}}^*$};
		\node[scale=.8,draw,thick,rounded corners,fill=white] at (1.2,-.8) {$\gamma^{(a)}$};
\end{tikzpicture}} \ = \ \raisebox{-9mm}{\begin{tikzpicture}
		\draw[->] (-.1,-.8) to (-.1,-1.5);
		\draw[in=-90,out=-90,looseness=2] (.4,-.8) to (1.1,-.8);
		\draw[<-] (1.1,-.8) to (1.1,1.5);
		\draw[->] (-.6,.9) to (-.6,1.5);
		\draw[in=90,out=90,looseness=2,->] (-1,.9) to (-1.7,.9);
		\draw (-1.7,.9) to (-1.7,-1.5);
		\draw (-.5,-.5) to (-.5,-1.5);
		\draw (0,.5) to (0,1.5);
		\draw[in=-90,out=90,looseness=2] (-.5,-.5) to (0,.5);
		\draw[white,line width=1mm,in=-90,out=90,looseness=2] (0,-.5 ) to (-.5,.5);
		\draw[red,in=-90,out=90,looseness=2] (0,-.5 ) to (-.5,.5);
		\draw[fill=white] (-.25,0) circle(2pt);
		\node[draw,thick,fill=white,rounded corners] at (0.2,-.6) {$\gamma^{(a)}$};
		\node[draw,thick,fill=white,rounded corners] at (-.8,.6) {${\gamma^{(b)}}^*$};
\end{tikzpicture}} \ = \ p_X  $. The second and the third equality follows from \Cref{Qsysgammaklem}. The fourth equality follows from \Cref{modificationeqns} for $m_a$'s and \Cref{Qdual}. Hence, $p_X$ is a projection.
\end{proof}

Since $\mcal D$ is idempotent complete, by \Cref{proj1cell} we have a $\widetilde X \in \mcal D_1(G(a),G(b))$ and $u_X \in \mcal D_2 \left(\widetilde X , x_b \us{F_b} \boxtimes F(X) \us{F_a} \boxtimes \ol{x}_a \right)$ such that $u_X^* u_X = 1_{\widetilde X}$ and $u_X u_X^* = p_X$ for each $a, b \in \mcal C_0$ and $X \in \mcal C_1(a,b) \ $.

Define $G(X) \coloneqq \widetilde X \in \mcal D_1(G(a),G(b)) $. 

\begin{lem}\label{gammamultiplication}
	For each $a \in \mcal C_0$
	\begin{itemize}
		\item [(i)]	$\raisebox{-14mm}
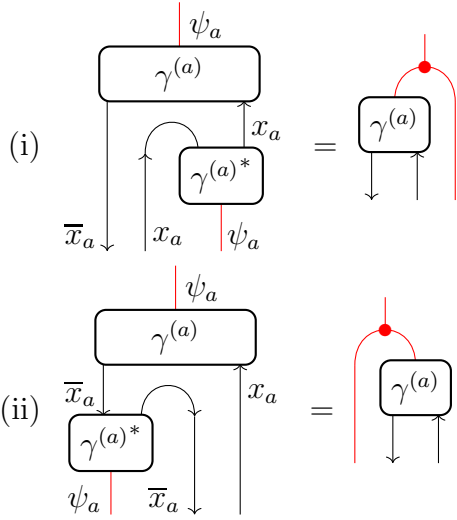
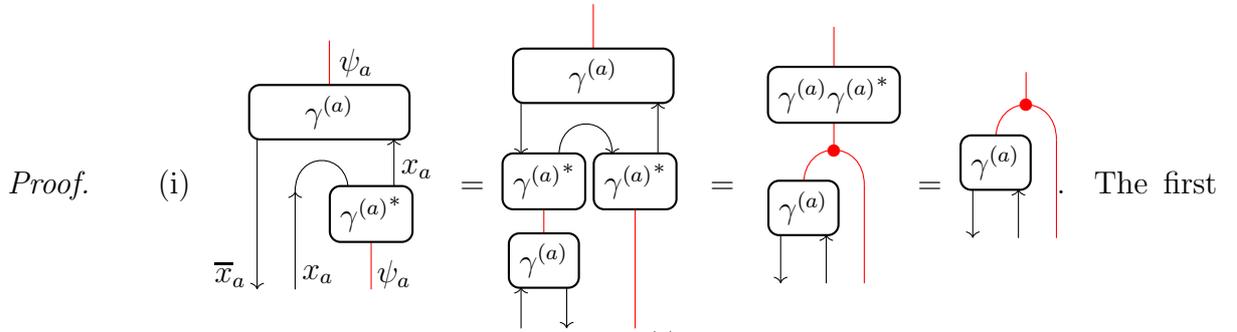
{\begin{tikzpicture}
		\draw[in=90,out=90,looseness=2] (-.3,.3) to (-1,.3);
		\draw[<-] (-1,.3) to (-1,-1);
		\draw[<-] (-1.5,-1) to (-1.5,1);
		\draw[->] (.3,.3) to (.3,1);
		\draw[red] (0,-.3) to (0,-1);
		\draw[red] (-.55,1.6) to (-.55,2.3);
		\node[draw,thick,rounded corners,fill=white] at (0,0) {${\gamma^{(a)}}^*$};
		\node[draw,thick,rounded corners,fill=white,minimum width=60] at (-.55,1.35) {${\gamma^{(a)}}$};
		\node at (.3,-.8) {$\psi_a$};
		\node at (-.7,-.8) {$x_a$};
		\node at (-1.85,-.8) {$\ol x_a$};
		\node at (-.2,2) {$\psi_a$};
		\node at (.6,.6) {$x_a$};
	\end{tikzpicture}} \ = \ \raisebox{-6mm}{\begin{tikzpicture}
\draw [->] (-.3,-.3) to (-.3,-1);
\draw [<-] (.3,-.35) to (.3,-1);
\draw[red] (.8,.3) to (.8,-1);
\draw[red] (.4,.7) to (.4,1.2);
\draw[red,in=90,out=90,looseness=2] (0,.3) to (.8,.3);
\node[draw,thick,rounded corners,fill=white] at (0,0) {$\gamma^{(a)}$};
\node at (.4,.75) {$\red{\bullet}$};
\end{tikzpicture}}  $
\item [(ii)] $\raisebox{-14mm}{\begin{tikzpicture}
		\draw[in=90,out=90,looseness=2,<-] (-.3,.3) to (-1,.3);
		\draw[->] (-.3,.3) to (-.3,-1);
		\draw[<-] (-1.5,.3) to (-1.5,1);
		\draw[->] (.3,-1) to (.3,1);
		\draw[red] (-1.4,-.3) to (-1.4,-1);
		\draw[red] (-.55,1.6) to (-.55,2.3);
		\node[draw,thick,rounded corners,fill=white] at (-1.4,-.05) {${\gamma^{(a)}}^*$};
		\node[draw,thick,rounded corners,fill=white,minimum width=60] at (-.55,1.35) {${\gamma^{(a)}}$};
		\node at (-1.8,.6) {$\ol x_a$};
		\node at (-.7,-.8) {$\ol x_a$};
		\node at (-1.75,-.8) {$\psi_a$};
		\node at (-.2,2) {$\psi_a$};
		\node at (.6,.6) {$x_a$};
\end{tikzpicture}} \ = \ \raisebox{-6mm}{\begin{tikzpicture}
		\draw [->] (.5,-.3) to (.5,-1);
		\draw [<-] (1.1,-.35) to (1.1,-1);
		\draw[red] (0,.3) to (0,-1);
		\draw[red] (.4,.7) to (.4,1.2);
		\draw[red,in=90,out=90,looseness=2] (0,.3) to (.8,.3);
		\node[draw,thick,rounded corners,fill=white] at (.8,0) {$\gamma^{(a)}$};
		\node at (.4,.75) {$\red{\bullet}$};
\end{tikzpicture}}  $

	\end{itemize}
\end{lem}

\begin{proof}
	\begin{itemize}
		\item [(i)] $\raisebox{-14mm}{\begin{tikzpicture}
			\draw[in=90,out=90,looseness=2] (-.3,.3) to (-1,.3);
			\draw[<-] (-1,.3) to (-1,-1);
			\draw[<-] (-1.5,-1) to (-1.5,1);
			\draw[->] (.3,.3) to (.3,1);
			\draw[red] (0,-.3) to (0,-1);
			\draw[red] (-.55,1.6) to (-.55,2.3);
			\node[draw,thick,rounded corners,fill=white] at (0,0) {${\gamma^{(a)}}^*$};
			\node[draw,thick,rounded corners,fill=white,minimum width=60] at (-.55,1.35) {${\gamma^{(a)}}$};
			\node at (.3,-.8) {$\psi_a$};
			\node at (-.7,-.8) {$x_a$};
			\node at (-1.85,-.8) {$\ol x_a$};
			\node at (-.2,2) {$\psi_a$};
			\node at (.6,.6) {$x_a$};
	\end{tikzpicture}} = \raisebox{-18mm}{\begin{tikzpicture}
	\draw[in=90,out=90,looseness=2,<-] (-.3,.3) to (-1,.3);
	\draw[<-] (-1.5,.3) to (-1.5,1);
	\draw[->] (.3,.3) to (.3,1);
	\draw[red] (0,-.3) to (0,-2);
	\draw[red] (-.55,1.6) to (-.55,2.3);
	\draw[red] (-1.2,-.3) to (-1.2,-.8);
	\draw[<-] (-1.5,-1.45) to (-1.5,-2);
	\draw[->] (-.9,-1.4) to (-.9,-2);
	\node[draw,thick,rounded corners,fill=white] at (-1.2,-1.1) {$\gamma^{(a)}$};
	\node[draw,thick,rounded corners,fill=white] at (0,-.05) {${\gamma^{(a)}}^*$};
	\node[draw,thick,rounded corners,fill=white] at (-1.2,-.05) {${\gamma^{(a)}}^*$};
	\node[draw,thick,rounded corners,fill=white,minimum width=60] at (-.55,1.35) {${\gamma^{(a)}}$};
	\end{tikzpicture}} \ = \ \raisebox{-12mm}{\begin{tikzpicture}
	\draw [->] (-.3,-.3) to (-.3,-1);
	\draw [<-] (.3,-.35) to (.3,-1);
	\draw[red] (.8,.3) to (.8,-1);
	\draw[red] (.4,.7) to (.4,2.4);
	\draw[red,in=90,out=90,looseness=2] (0,.3) to (.8,.3);
	\node[draw,thick,rounded corners,fill=white] at (0,0) {$\gamma^{(a)}$};
	\node at (.4,.75) {$\red{\bullet}$};
	\node[draw,thick,rounded corners,fill=white] at (.4,1.5) {$\gamma^{(a)} {\gamma^{(a)}}^*$};
\end{tikzpicture}} = \raisebox{-6mm}{\begin{tikzpicture}
\draw [->] (-.3,-.3) to (-.3,-1);
\draw [<-] (.3,-.35) to (.3,-1);
\draw[red] (.8,.3) to (.8,-1);
\draw[red] (.4,.7) to (.4,1.2);
\draw[red,in=90,out=90,looseness=2] (0,.3) to (.8,.3);
\node[draw,thick,rounded corners,fill=white] at (0,0) {$\gamma^{(a)}$};
\node at (.4,.75) {$\red{\bullet}$};
\end{tikzpicture}}  $. The first equality follows from the unitarity of $\gamma^{(a)}$. The second equality is a consequence of \Cref{Qsysgammakcor}. Finally, the third equality is because of unitarity of $\gamma^{(a)}$.
\item [(ii)] The proof is similar to that of (i)

	\end{itemize}
\end{proof}

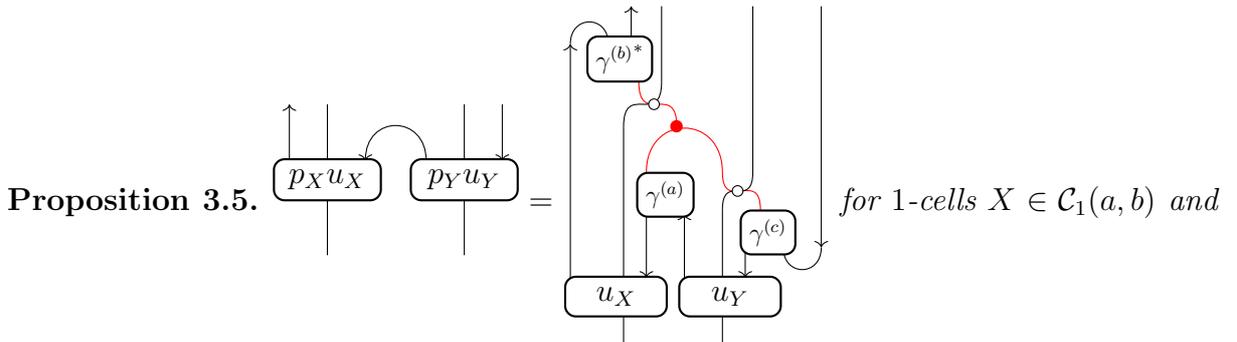
\begin{prop}\label{projprop}
	$\raisebox{-6mm}{\begin{tikzpicture}
			\draw [in=90,out=90,looseness=2,->] (1.3,.25) to (.5,.25);
			\draw (0,-1) to (0,1);
			\draw (1.8,-1) to (1.8,1);
			\draw[->] (-.5,0) to (-.5,1);
			\draw[<-] (2.3,.25) to (2.3,1);
			\node[draw,thick,fill=white,rounded corners,minimum width=40] at (0,0) {$p_X u_X$};
			\node[draw,thick,fill=white,rounded corners,minimum width=40] at (1.8,0) {$p_Y u_Y$};
	\end{tikzpicture}} = \raisebox{-18mm}{\begin{tikzpicture}
			\draw[red,in=90,out=90,looseness=2] (0,.3) to (1,.3);
			\draw[red,in=90,out=-90,looseness=2] (1,.3) to (1.5,-.2);
			\draw[red,in=-90,out=90,looseness=2] (.4,.9) to (-.1,1.5);
			\draw[->] (0,-.25) to (0,-1.1);
			\draw[<-] (.5,-.25) to (.5,-1.1);
			\draw (-.3,.9) to (-.3,-2);
			\draw [in=-90,out=90,looseness=2] (-.3,.9) to (.2,1.5);
			\draw [in=-90,out=90,looseness=2] (1,-.2) to (1.4,.3);
			\draw (1,-.2) to (1,-2);
			\draw (1.4,.3) to (1.4,2.5);
			\draw (.2,1.5) to (.2,2.5);
			\draw[fill=white] (.1,1.2) circle(2pt);
			\draw[fill=white] (1.2,.05) circle(2pt);
			\draw [in=90,out=90,looseness=2] (-.5,2) to (-1,2);
			\draw[<-] (-1,2) to (-1,-1.1);
			\draw[->] (-.2,2) to (-.2,2.5);
			\draw[in=-90,out=-90,looseness=2] (1.8,-.7) to (2.3,-.7);
			\draw[->] (2.3,2.5) to (2.3,-.7);
			\draw[->] (1.3,-.7) to (1.3,-1.1); 
			\node[scale=.8,draw,thick,rounded corners,fill=white] at (.25,0) {$\gamma^{(a)}$};
			\node[scale=.8,draw,thick,rounded corners,fill=white] at (-.35,1.8) {${\gamma^{(b)}}^*$};
			\node[scale=.8,draw,thick,rounded corners,fill=white] at (1.6,-.5) {$\gamma^{(c)}$};
			\node at (.4,.9) {$\red{\bullet}$};
			\node[draw,thick,rounded corners,fill=white,minimum width=38] at (-.4,-1.35) {$u_X$};
			\node[draw,thick,rounded corners,fill=white,minimum width=38] at (1.1,-1.35) {$u_Y$};	
	\end{tikzpicture}} $ for $ 1 $-cells $X \in \mcal C_1(a,b)$ and $Y \in \mcal C_1(c,a)$.
\end{prop}
\begin{proof}
	The last equality is an easy application of \Cref{gammamultiplication}, \Cref{Qsysgammakcor}, the fact that $m$ is a $ 2 $-modification, and separability and Frobenius condition of the Q-systems $\psi_a$'s .
\end{proof}

\subsection{$G$ on $ 2 $-cells}\

Suppose $X , Y \in \mcal C_1 (a,b)$ and $f \in \mcal C_2(X,Y)$.

Define $G(f) \coloneqq \raisebox{-22mm}{\begin{tikzpicture}
	\draw (0,2) to (0,-2);
	\draw[<-] (-1,.8) to (-1,-.8);
	\draw[->] (1,.8) to (1,-.8);
	\node[draw,thick,rounded corners,fill=white] at (0,0) {$F(f)$};
	\node[draw,thick,rounded corners,fill=white,minimum width=65] at (0,1.1) {$u_Y^*$};
	\node[draw,thick,rounded corners,fill=white,minimum width=65] at (0,-1.05) {$u_X$};
	\node[scale=.8] at (.5,-1.8) {$G(X)$};
	\node[scale=.8] at (.5,1.8) {$G(Y)$};
	\node at (-1.4,0) {$x_b$};
	\node at (1.4,0) {$\ol x_a$};	
\end{tikzpicture}} \in \mcal D_2 (G(X),G(Y))$, where $u_X$ and $u_Y$ are the isometries splitting the projections $p_X$ and $p_Y$ respectively.

\subsection{Tensorators and Unitors of $G$}\

We proceed to define the tensorators $G_{X,Y}^2 \in \mcal D_2 \left( G(X) \us{G(a)}\boxtimes G(Y) , G \left(X \us{a}\boxtimes Y \right) \right)$ for each $X \in \mcal C_1(a,b)$ and $Y \in \mcal C_1(c,a) \ $.

We have $u_X : G(X) \to x_b \us{F_b} \boxtimes F(X) \us{F_a} \boxtimes \ol{x}_a $, $u_Y : G(Y) \to x_a \us{F_a} \boxtimes F(Y) \us{F_c} \boxtimes \ol{x}_c $, and $u_{XY} : G \left(X \us{G(a)}\boxtimes Y \right) \to x_b \us{F_b}\boxtimes F (X \boxtimes Y) \us{F_c}\boxtimes \ol x_c $, where $u_X$, $u_Y$, and $u_{XY}$ are the isometries corresponding to the projections $p_X$, $p_Y$, and $p_{X \boxtimes Y}$ respectively. 

Define $G_{X,Y}^2 \coloneqq \raisebox{-2.4cm}{
}$ for $ 2 $-cells $f \in \mcal C_2(_aX_b, \, _aY_b)$
	\end{itemize}
\end{lem}
\begin{proof}\
	
	\begin{itemize}
		\item [(i)] The proof follows from \Cref{projprop} that $p_X = p_X^*$ for every $ 1 $-cell $X$ and the fact $\psi : F \Rightarrow F$ is a $ 2 $-transformation.
		
		\item[(ii)] The proof is similar to that of (i)
	\end{itemize}
\end{proof}

\begin{prop}\label{Gunitary}
	For each $1$-cells $X$ and $Y$ in $\mcal C$, $G_{X,Y}^2$ is an unitary.
\end{prop}
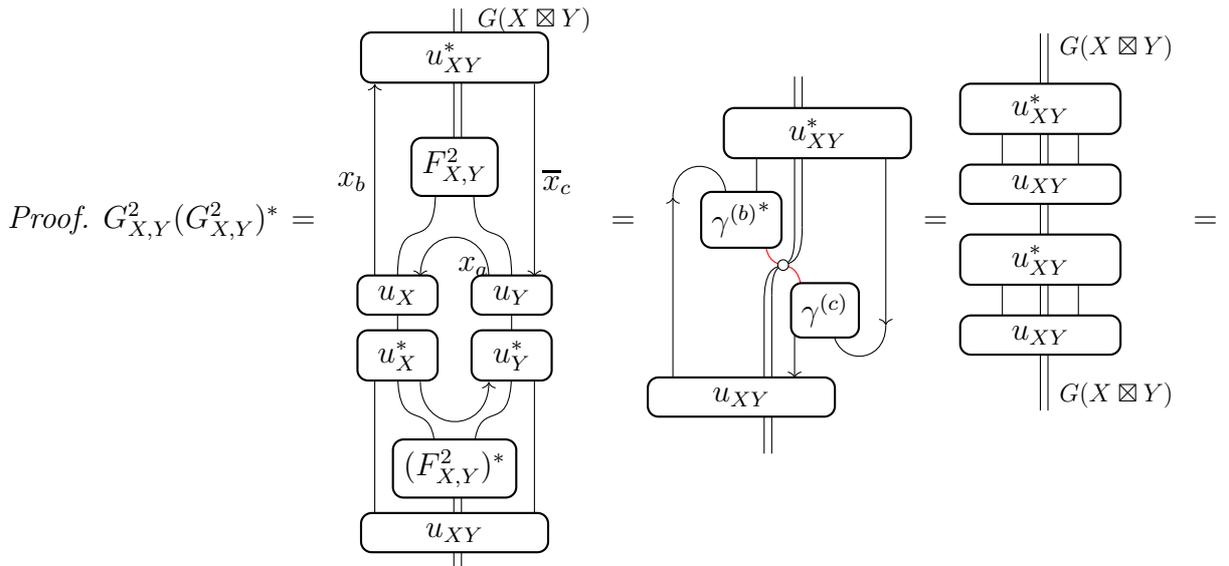
\begin{proof}
	$G_{X,Y}^2 (G_{X,Y}^2)^* = \raisebox{-4.5cm}{%
}  $. The second equality follows from the definition of $p_X$ and $p_Y$ and using the unitarity of $F_{X,Y}^2$. The last equality uses the fact that, $u_{XY}$ is an isometry.
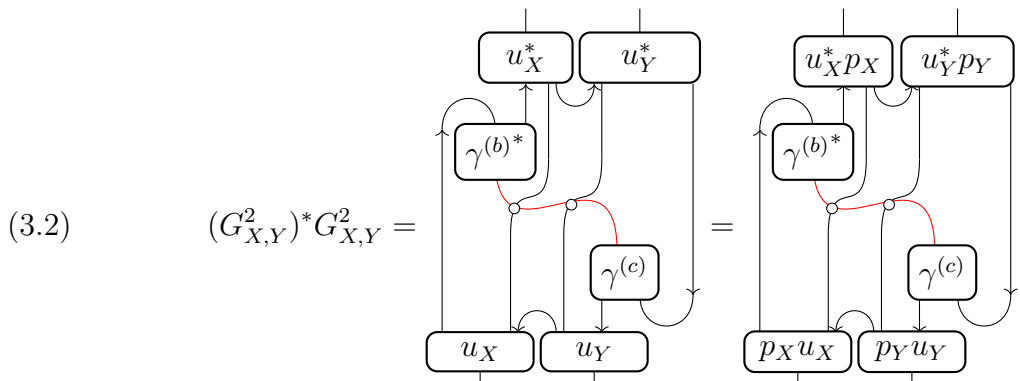
\begin{equation}\label{tensoreqn}
(G_{X,Y}^2)^* G_{X,Y}^2 = \raisebox{-2cm}{%
}
\end{equation} 
Therefore, using the fact that $m$ is a 2-modification(\Cref{modificationeqns}), separabilty, and Frobenius condition of $\psi_a$'s and \Cref{projprop},  $\raisebox{-2cm}{%
} $. Using the last equality in \Cref{tensoreqn} we get, $(G_{X,Y}^2)^* G_{X,Y}^2 = \raisebox{-36mm}{%
}$

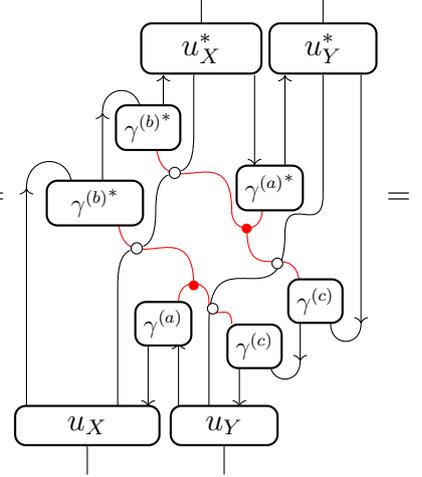
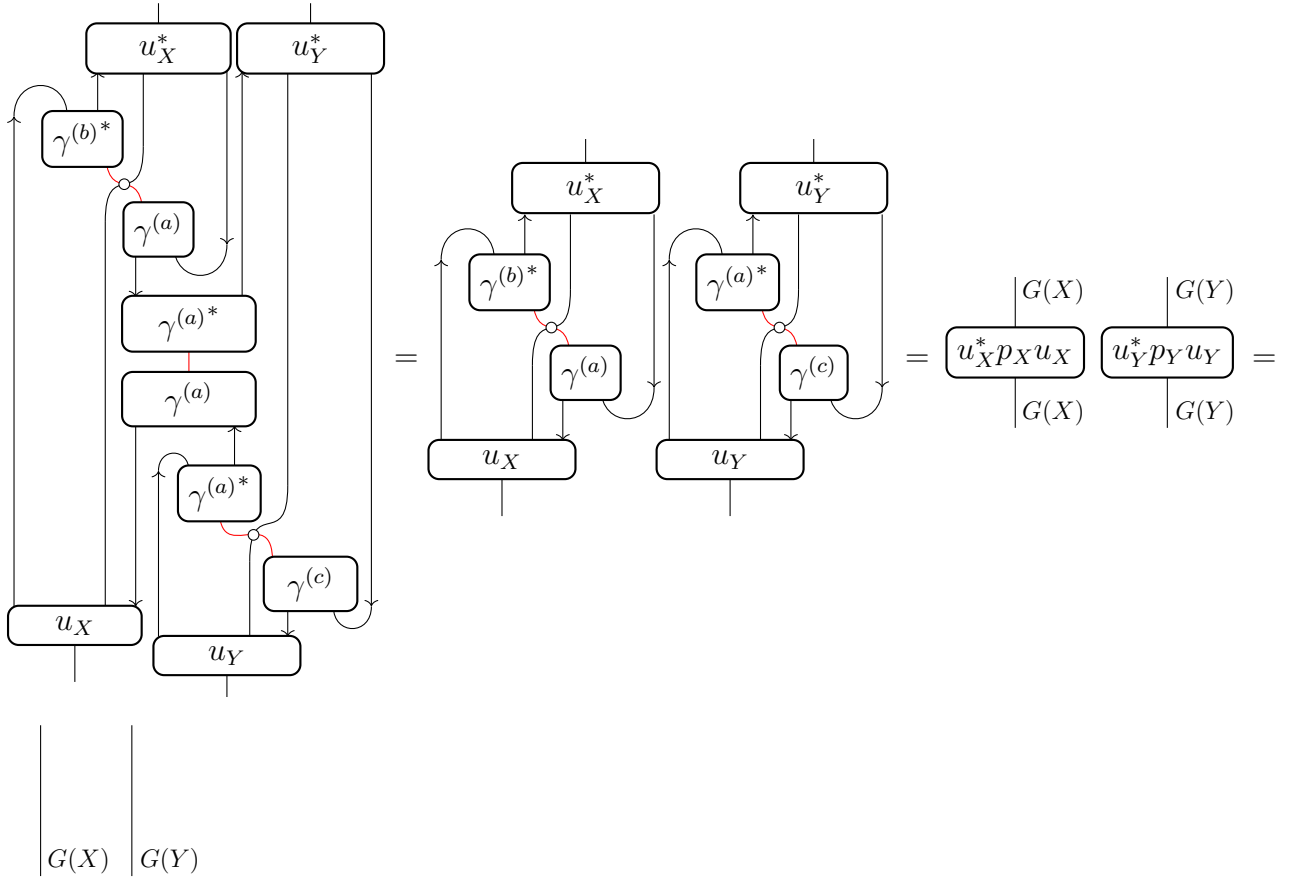
\comments{(\red{To add one picture after 2nd equality}) }The second equality follows from \Cref{Qsysgammakcor}, Frobenius condition and \Cref{gammamultiplication}. This concludes the proposition.
\end{proof}

For each $a \in \mcal C_0$, consider the projection $p_{1_a} = \raisebox{-11mm}{%
} $. Using the fact that $F_a^1$ is a unitary and $\psi$ being a $ * $-$ 2 $-transformation, $p_{1_a} = \raisebox{-11mm}{%
} $. Since the isometry $\raisebox{-8mm}{%
}$ splits the projection $p_{1_a}$, we define $G(1_a) \coloneqq 1_{G(a)}$, $G^1_a \coloneqq \t{id}_{1_{G(a)}}$, and $u_{1_a} = \raisebox{-8mm}{%
}$. The last equality is an easy application of the \Cref{newxrelF}(i).
Also in a similar way, $\raisebox{-18mm}{%
}$. Thus, the result will follow from the following assertion and the fact that $F$ is $ 2 $-functor. \\
\textit{Assertion :} $\raisebox{-20mm}{%
}$ \\

\noindent
\textit{Proof of the assertion:} Left hand side of the assertion becomes, \begin{equation}\label{LHSeqn}
		\raisebox{-20mm}{%
}
\end{equation}
The first equality follows from \Cref{projprop} and the second equality follows from the definition of $p_Z$ and \Cref{gammamultiplication}. Similarly right hand side of the assertion, becomes the last term of \Cref{LHSeqn}. This completes the proof of the assertion.

\vspace*{2mm}

\noindent
(ii) $\raisebox{-12mm}{%
} = u_X^* u_X = 1_{G(X)}$. The second equality is the definition of $G_{X,1_a}^2$. The third equality is an application of the unitality condition of the $ 2 $-functor $F$. In a similar way one obtains,
$\raisebox{-12mm}{%
} = 1_{G(X)}$. 
\end{proof}

\Cref{Gunitary} and \Cref{Gtensor} settles the following proposition.

\begin{prop}
	$ G : \mcal C \to \mcal D $ is a $ * $-$ 2 $-functor.
\end{prop}

\section{Construction of a dualizable $ 1 $-cell in $\textbf{Fun}(\mcal C, \mcal D)$}\label{2transformation}

In this section, we construct a suitable dualizable $ 2 $-transformation $\phi$ from $F$ to $G$. We will follow the notations as in \Cref{2functor}.

\vspace*{2mm}

For $a, \ b  \in \mcal C_0$ and $_bX_a \in \mcal C_1(a,b)$, define $$\phi_a \coloneqq x_a \in \mcal D_1(Fa,Ga) \ \t{and} \ \phi_X \coloneqq 
\raisebox{-12mm}{
} \in \mcal D_2 \left(\phi_b \boxtimes F(X) , G(X) \boxtimes \phi_a \right)	
	$$
\begin{lem}\label{phiunitary}
	For each $ 1 $-cell $X \in \mcal C_1(a,b)$, $\phi_X$ is an unitary.
\end{lem}
\begin{proof}
	$\phi_X^* \phi_X = \raisebox{-14mm}{%
}$. The fourth equality follows from the fact that the unit map $i : 1_F \Rrightarrow \psi$ is a $ 2 $-modification and second part of \Cref{gammaisoeqn}. An easy application of \Cref{Qsysgammaklem}(ii) reveals the last equality.

$\phi_X \phi_X^* = \raisebox{-14mm}{%
}$ (This is because $u_X u_X^* = p_X$ and $u_X^* u_X = 1_{G(X)}$)

Now $\raisebox{-4mm}{%
} \ \ $ (The last equality follows from (i) of \Cref{Qsysgammaklem}) and $ \ \ \raisebox{-4mm}{%
} $. The second equality follows from \Cref{Qsysgammakcor} and the fact that $m$ is a $ 2 $-modification. The third equality will follow from Frobenius condition and unitality of $(\psi_\bullet,m_\bullet,i_\bullet)$. The fourth equality again follows from \Cref{Qsysgammakcor}.
So, $\phi_X \phi_X^* \ = \ \raisebox{-14mm}{%
}$

Hence, each $\phi_X$ is a unitary
\end{proof}
We will use the following pictorial notations :
\[\phi_X \ = \ \raisebox{-8mm}{%
}  $. The first equality follows from the definition of $G_{X,Y}^2$, $\phi_X$, and $\phi_Y$. The second equality is evident from \Cref{newxrelF}(i). The definition of $p_{X \boxtimes Y}$ reveals the third equality. The last equality follows from the definition of $\phi_{X \boxtimes Y}$

\item [(ii)] $\raisebox{-12mm}{%
}$. The first equality follows from the definition of $G(f)$ and $\phi_X$. Since $u_X u_X^* = p_X$, we have the second equality. The third equality follows from \Cref{newxrelF}(ii) and the fourth equality from the definition of $p_Y$. An easy observation of the definition of $\phi_Y$ reveals the last equality.

\item [(iii)] Since, $G^1_a \coloneqq \t{id}_{1_{G(a)}}$, and $u_{1_a} = \raisebox{-8mm}{%
}$, a straightforward verification concludes (iii).
	\end{itemize}
\end{proof}

\begin{cor}
	$\phi : F \Rightarrow G$ is a $ * $-$ 2 $-transformation.
\end{cor}

\subsection{Dualizability of $\phi$ in $\textbf{Fun}(\mcal C, \mcal D)$}\

We construct a dual $\ol \phi : G \Rightarrow F$ of the $ * $-$ 2 $-transformation $\phi : F \Rightarrow G$

\vspace*{2mm}

For $a, \ b  \in \mcal C_0$ and $_bX_a \in \mcal C_1(a,b)$, define $$ \ol \phi_a \coloneqq \ol x_a \in \mcal D_1(Ga,Fa) \ \t{and} \ \ol \phi_X \coloneqq 
\raisebox{-12mm}{%
} \in \mcal D_2 \left(\ol \phi_b \boxtimes G(X) , F(X) \boxtimes \ol \phi_a \right).	
$$

\begin{lem}
	For each $X \in \mcal C_1(a,b)$, $\ol \phi_X$ is an unitary
\end{lem}
\begin{proof}
	The proof is similar to that of \Cref{phiunitary}.
\end{proof}

We will use the following pictorial notations :
\[\ol \phi_X \ = \ \raisebox{-8mm}{%
}\]
\begin{rem}\label{phidualrem}
	From the definitions of $\phi_X$ and $\ol \phi_X$ it follows that, $\raisebox{-8mm}{%
}$ for every $ 1 $-cell $X \in \mcal C_1(a,b)$.
\end{rem}

\begin{prop}
	$\ol \phi : G \Rightarrow F$ is a $ * $-$ 2 $-transformation.
\end{prop}
\begin{proof}
	The proof follows along the same lines as in \Cref{phi2transformtion}
\end{proof}

\begin{prop}
	$\phi : F \Rightarrow G$ has a unitary separable dual $\ol \phi : G \Rightarrow F$ in $\normalfont\textbf{Fun}(\mcal C, \mcal D)$
\end{prop}
\begin{proof}
	We need to show that, for each $a \in \mcal C_0$,  $\raisebox{-4mm}{%
}$. The last equality follows from second part of \Cref{gammaisoeqn}, the fact that $i : 1_F \Rrightarrow \psi$ is a $ 2 $-modification, and \Cref{Qsysgammaklem}(ii)

Also, $\raisebox{-8mm}{%
} $. The first equality follows from \Cref{phidualrem}. Hence, we have the result 
\end{proof}

\section{Isomorphism of Q-systems}\label{2modification}
In this section we apply the results from previous sections to build an unitary between the $ 1 $-cells $\left(\psi_\bullet,m_\bullet,i_\bullet\right)$ and $\ol \phi \otimes \phi$ in $\textbf{Fun}(\mcal C,\mcal D)$. We show that $\left\{\gamma^{(a)} : \ol \phi_a \boxtimes \phi_a \to \psi_a \right\}_{a \in \mcal C_0}$ is an unitary $ 2 $-modification intertwining the algebra maps.

\begin{prop}\label{gammaalg}
	$\left\{\gamma^{(a)} : \ol \phi_a \boxtimes \phi_a \to \psi_a \right\}_{a \in \mcal C_0}$ intertwines the algebra maps.
\end{prop}
\begin{proof}
	This is evident from \Cref{gammaisoeqn}
\end{proof}

\begin{prop}\label{gammamodification}
	$\left\{\gamma^{(a)} : \ol \phi_a \boxtimes \phi_a \to \psi_a \right\}_{a \in \mcal C_0}$ is a $ 2 $-modification
\end{prop}
\begin{proof}
	For each $a \in C_0$, $\gamma^{(a)} \in \mcal D_2 (\ol \phi_a \boxtimes \phi_a , \psi_a)$ is an unitary. Hence, $\us{a \in \mcal C_0}{\t{sup}} \norm{\gamma^{(a)}} = 1$. Thus, in order to complete the proof we need to show that each $\gamma^{(a)}$ satisfy \Cref{modificationeqns}. Now, $\raisebox{-16mm}{\begin{tikzpicture}
			\draw[in=-90,out=90,looseness=2,->] (.6,0) to (1.1,1);
			\draw[in=-90,out=90,looseness=2,<-] (0,0) to (.5,1);
			\draw[in=-90,out=90,looseness=2] (1.1,0) to (0,1);
			\draw[fill=white] (.8,.5) circle(2pt);
			\draw[fill=white] (.3,.5) circle(2pt);
			\draw[red] (.8,1.6) to (.8,2.3);
			\draw (0,1) to (0,2.3); 
			\node at (-.2,-.1) {$\ol \phi_b$};
			\node at (.8,-.2) {$\phi_b$};
			\node[scale=.8] at (1.6,0) {$F(X)$};
			\node[scale=.8] at (-.5,1.8) {$F(X)$};
			\node at (1.5,.8) {$\phi_a$};
			\node at (1.1,2.1) {$\psi_a$};
			\node[draw,thick,rounded corners,fill=white] at (.8,1.35) {$\gamma^{(a)}$};
			\end{tikzpicture}} = \raisebox{-16mm}{\begin{tikzpicture}
			\draw (0,1.8) to (0,-1);
			\draw[<-] (.3,.25) to (.3,1);
			\draw (-1,.2) to (-1,-1);
			\draw[in=90,out=90,looseness=2,->] (-.3,.2) to (-1,.2);
			\draw[in=-90,out=-90,looseness=2,->] (.3,-.4) to (1.1,-.4);
			\draw (1.1,-.4) to (1.1,1);
			\draw[red] (.7,1.4) to (.7,1.8);
			\draw[<-] (-.3,-.45) to (-.3,-1);	
			\node[draw,thick,rounded corners,fill=white,minimum width=35] at (0,-.1) {$u_X u_X^*$};
			\node[scale=.8] at (.2,-1.2) {$F(X)$};
			\node[scale=.8] at (-.4,1.5) {$F(X)$};
			\node at (-1.25,-.8) {$\ol \phi_b$};
			\node at (-.5,-.8) {$\phi_b$};
			\node at (.6,.5) {$\ol \phi_a$};
			\node at (1.4,.5) {$\phi_a$};
			\node at (1,1.8) {$\psi_a$};
			\node[draw,thick,rounded corners,fill=white] at (.7,1.2) {$\gamma^{(a)}$}; 	
		\end{tikzpicture}} = \raisebox{-16mm}{\begin{tikzpicture}
		\draw (0,1.8) to (0,-1);
		\draw[<-] (.3,.25) to (.3,1);
		\draw (-1,.2) to (-1,-1);
		\draw[in=90,out=90,looseness=2,->] (-.3,.2) to (-1,.2);
		\draw[in=-90,out=-90,looseness=2,->] (.3,-.4) to (1.1,-.4);
		\draw (1.1,-.4) to (1.1,1);
		\draw[red] (.7,1.4) to (.7,1.8);
		\draw[<-] (-.3,-.45) to (-.3,-1);	
		\node[draw,thick,rounded corners,fill=white,minimum width=35] at (0,-.1) {$p_X^*$};
		\node[scale=.8] at (.2,-1.2) {$F(X)$};
		\node[scale=.8] at (-.4,1.5) {$F(X)$};
		\node at (-1.25,-.8) {$\ol \phi_b$};
		\node at (-.5,-.8) {$\phi_b$};
		\node at (.6,.5) {$\ol \phi_a$};
		\node at (1.4,.5) {$\phi_a$};
		\node at (1,1.8) {$\psi_a$};
		\node[draw,thick,rounded corners,fill=white] at (.7,1.2) {$\gamma^{(a)}$}; 	
	\end{tikzpicture}} = \raisebox{-12mm}{\begin{tikzpicture}
		\draw[in=-90,out=90,looseness=2] (.5,0) to (0,1);
		\draw (0,1) to (0,1.7);
		\draw (.5,0) to (.5,-.8);
		\draw[white,line width=1mm,in=-90,out=90,looseness=2] (0,0) to (.5,1);
		\draw[red,in=-90,out=90,looseness=2] (0,0) to (.5,1);
		\draw[red] (.7,1.3) to (.7,1.7);
		\draw[fill=white] (.25,.5) circle(2pt);
		\draw[<-] (.1,-.55) to (.1,-1.1);
		\draw[->] (-.3,-.55) to (-.3,-1.1);
		\node[draw,thick,rounded corners,fill=white] at (-.1,-.2) {$\gamma^{(b)}$};
		\node[scale=.8,draw,thick,rounded corners,fill=white] at (.8,1) {$\gamma^{(a)} {\gamma^{(a)}}^*$};
	\end{tikzpicture}} = \raisebox{-12mm}{\begin{tikzpicture}
	\draw[in=-90,out=90,looseness=2] (.5,0) to (0,1);
	\draw (0,1) to (0,1.7);
	\draw (.5,0) to (.5,-.8);
	\draw[white,line width=1mm,in=-90,out=90,looseness=2] (0,0) to (.5,1);
	\draw[red,in=-90,out=90,looseness=2] (0,0) to (.5,1);
	\draw[red] (.5,1) to (.5,1.7);
	\draw[fill=white] (.25,.5) circle(2pt);
	\draw[<-] (.1,-.55) to (.1,-1.1);
	\draw[->] (-.3,-.55) to (-.3,-1.1);
	\node[draw,thick,rounded corners,fill=white] at (-.1,-.2) {$\gamma^{(b)}$};
	\node[scale=.8] at (1,-.6) {$F(X)$};
	\node[scale=.8] at (-.5,1.5) {$F(X)$};
	\node at (.8,1.5) {$\psi_a$};
	\node at (-.6,-1) {$\ol \phi_b$};
	\node at (.35,-1) {$\phi_b$};
\end{tikzpicture}}$. The first equality is an easy application of definitions of $\phi_X$ and $\ol \phi_X$. Since $u_X u_X^* = p_X = p_X^*$, we have the second equality. The last equality follows from the unitarity of $\gamma^{(a)}$'s. Thus, $\gamma : \ol \phi \otimes \phi \Rrightarrow \psi$ is a unitary $ 2 $-modification 
\end{proof}

Using \Cref{gammaalg} and \Cref{gammamodification} we get the following

\begin{thm}\label{Qsysiso}
	$\left(\psi_\bullet, m_\bullet,i_\bullet\right)$ and $\ol \phi \otimes \phi$ are isomorphic as Q-systems in ${\normalfont\t{End}_{\textbf{Fun}(\mcal D, \mcal D)}}(F)$
\end{thm}

This settles \Cref{maintheorem} by using \cite[Theorem 3.36]{CPJP}

\begin{cor}
	Suppose $\mcal D$ is an idempotent complete $C^*$-$ 2 $-category. We have the following.
	\begin{itemize}
		\item [(i)] $\mcal D$ is Q-system complete if and only if $\normalfont \textbf{Fun}(\mcal D,\mcal D)$ is Q-system complete
		
		\item [(ii)] $\mcal D$ is Q-system complete if and only if for every $C^*$-$ 2 $-category $\mcal C$, $\normalfont \textbf{Fun}(\mcal C, \mcal D)$ is Q-system complete 
	\end{itemize}
\end{cor}

\begin{proof}\
	
	\begin{itemize}
		\item [(i)] If $\mcal D$ is Q-system complete, then obviously by \Cref{maintheorem}, $\normalfont \textbf{Fun}(\mcal D,\mcal D)$ is Q-system complete.
		
		Conversely, suppose $\normalfont \textbf{Fun}(\mcal D,\mcal D)$ is Q-system complete and $\left(_bQ_b,m,i\right)$ be a Q-system in $\mcal D$. We show that $Q$ `splits' in $\mcal D$. First we define a $ 2 $-functor $F : \mcal D \to \mcal D $ as follows.
		\[F(a) \coloneqq b \ \  \forall a \in \mcal D_0 \ , \ F(X) \coloneqq 1_b \ \ \forall X \in \mcal D_1(c,d) \ , \ F(f) \coloneqq \t{id}_{1_b} \ \ \forall f \in \mcal D_2(X,Y)  \]
		We now define a $ 2 $-transformation $\psi : F \Rightarrow F$  and $ 2 $-modifications $\widetilde m : \psi \boxtimes \psi \Rrightarrow \psi$ and $\widetilde i : 1_F \Rrightarrow \psi$ as follows.
		\[\psi_a \coloneqq \,_bQ_b \ \ \forall a \in \mcal D_0 \ , \ \psi_X \coloneqq 1_Q \]
		\[ \widetilde m_a \coloneqq m \ , \ \widetilde i_a \coloneqq i \ \ \t{for all} \ \ a \in \mcal D_0 \]
		It is straightforward to verify that $\psi : F \Rightarrow F$ is a $ 2 $-transformation and $\widetilde m$ and $\widetilde i$ are $ 2 $-modifications. Also it is easy to see that $\left(\psi, \widetilde m, \widetilde i\right)$ is a Q-system in ${\normalfont\t{End}_{\textbf{Fun}(\mcal D, \mcal D)}}(F)$. Now, since $\textbf{Fun}(\mcal D, \mcal D)$ is Q-system complete, we have a $ 2 $-functor $G : \mcal D \to \mcal D$ and $ 2 $-transformation $\phi : F \Rightarrow G$ having a unitary separable dual such that, $\psi$ is isomorphic to $\ol \phi \boxtimes \phi$ as Q-systems in ${\normalfont\t{End}_{\textbf{Fun}(\mcal D, \mcal D)}}(F)$.
		Thus, $\psi_a \simeq \ol \phi_a \us{G(a)} \boxtimes \phi_a$ for every $a \in \mcal D_0$. In particular, $_bQ_b \simeq _b \ol \phi \us{G(b)} \boxtimes \phi_b $ as Q-systems. Thus, every Q-system in $\mcal D$ splits. Therefore, by \cite[Theorem 3.36]{CPJP} $\mcal D$ is Q-system complete
		
		\item [(ii)] An easy application of \Cref{maintheorem} and (i) concludes (ii)  
	\end{itemize}
\end{proof}
\begin{rem}
	Although we have assumed $\mcal C$ and $\mcal D$ to be strict C*-2-categories, \Cref{Qsysiso} and hence, \Cref{maintheorem} will also be true if the strictness condition is dropped. The proof is almost verbatim and involves repeated application of \Cref{projprop}
\end{rem}

We refer the reader to \cite[Definition 3.2]{CPJ} for the details of the $ 2 $-category $\textbf{C*Alg}_{\mcal C}$ of actions of $\mcal C$ on C*-algebras, where $\mcal C$ is a unitary fusion category.
\begin{cor}
	Suppose $\mcal C$ is a unitary tensor category. The $C^*$-$ 2 $-category of $\mcal C$ action on $C^*$-algebras $ \normalfont\textbf{C*Alg}_{\mcal C}$ is Q-system complete 
\end{cor}

\begin{proof}
	The proof follows from the observation \cite[Remark 3.3]{CPJ} that $\textbf{C*Alg}_{\mcal C} = \textbf{Fun}(\t{B}\mcal C, \textbf{C*Alg})$  and \Cref{maintheorem}, where $\t{B} \mcal C$ is the $ 2 $-category with one object and whose endomorphisms are the category $\mcal C$ and $\textbf{C*Alg}$ is C*-2-category of C*-algebras, bimodules and intertwiners
\end{proof}

\comments{\draw[red,in=-90,out=90,looseness=2] (0,0) to (2,2);
	\draw[in=-90,out=90,looseness=2] (.7,0) to (.3,2);
	\draw[in=-90,out=90,looseness=2] (1.5,0) to (1.1,2);
	\draw[fill=white] (.5,1) circle(2pt);
	\draw[fill=white] (1.3,1) circle(2pt);
	\draw[in=90,out=90,looseness=2,<-] (.9,0) to (1.3,0);}

\Contact


\end{document}